\documentclass[12pt]{amsart}
\usepackage{amssymb,amsfonts,amsmath,amsopn,amstext,amscd,latexsym,xy,color,graphicx,appendix}
\theoremstyle{plain}
\usepackage{hyperref}
\input xy
\xyoption{all}

\newtheorem{theorem}{Theorem}[section]
\newtheorem{lemma}[theorem]{Lemma}

\newtheorem{proposition}[theorem]{Proposition}
\newtheorem{corollary}[theorem]{Corollary}
\newtheorem{definition}[theorem]{Definition}

\theoremstyle{remark}

\newtheorem*{remark}{Remark}

\DeclareMathOperator{\GL}{GL}

\DeclareMathOperator{\depth}{depth}
\DeclareMathOperator{\Supp}{Supp}
\DeclareMathOperator{\Def}{Def}
\DeclareMathOperator{\Nil}{Nil}
\DeclareMathOperator{\Pol}{Pol}

\DeclareMathOperator{\Frob}{Frob}
\DeclareMathOperator{\Hom}{Hom}

\DeclareMathOperator{\rank}{rank}

\DeclareMathOperator{\gr}{gr}
\DeclareMathOperator{\Gal}{Gal}

\DeclareMathOperator{\ad}{ad}
\DeclareMathOperator{\diag}{diag}

\DeclareMathOperator{\End}{End}
\DeclareMathOperator{\tr}{tr}
\DeclareMathOperator{\pr}{pr}
\DeclareMathOperator{\Lie}{Lie}

\DeclareMathOperator{\cha}{char}
\DeclareMathOperator{\Res}{Res}

\DeclareMathOperator{\Ind}{Ind}
\DeclareMathOperator{\nInd}{n-Ind}

\DeclareMathOperator{\Fil}{Fil}

\DeclareMathOperator{\Art}{Art}
\DeclareMathOperator{\St}{Sp}
\DeclareMathOperator{\Iw}{Iw}
\DeclareMathOperator{\rec}{rec}
\DeclareMathOperator{\Spec}{Spec}

\newcommand{\cA}{{\mathcal A}}

\newcommand{\cC}{{\mathcal C}}
\newcommand{\cD}{{\mathcal D}}

\newcommand{\cG}{{\mathcal G}}
\newcommand{\cH}{{\mathcal H}}

\newcommand{\cL}{{\mathcal L}}
\newcommand{\cM}{{\mathcal M}}

\newcommand{\cO}{{\mathcal O}}

\newcommand{\cS}{{\mathcal S}}
\newcommand{\cT}{{\mathcal T}}

\newcommand{\fra}{{\mathfrak a}}

\newcommand{\frf}{{\mathfrak f}}

\newcommand{\frl}{{\mathfrak l}}
\newcommand{\ffrm}{{\mathfrak m}}

\newcommand{\frp}{{\mathfrak p}}

\newcommand{\bbA}{{\mathbb A}}

\newcommand{\bbC}{{\mathbb C}}

\newcommand{\bbF}{{\mathbb F}}

\newcommand{\bbN}{{\mathbb N}}

\newcommand{\bbQ}{{\mathbb Q}}
\newcommand{\bbR}{{\mathbb R}}

\newcommand{\bbT}{{\mathbb T}}

\newcommand{\bbZ}{{\mathbb Z}}
\newcommand{\ilim}{\mathop{\varinjlim}\limits}
\newcommand{\plim}{\mathop{\varprojlim}\limits}
\newcommand{\wv}{{\widetilde{v}}}

\newcommand{\barQl}{{\overline{\bbQ}_l}}

\author[J. Thorne]{Jack Thorne}
\address{Department of Mathematics, Harvard University, 1 Oxford Street, Cambridge, MA 02138, USA}
\email{thorne@math.harvard.edu}

\title[Representations with small residual image]{On the automorphy of $l$-adic Galois representations with small residual image}

\begin{document}

\begin{abstract}We prove new automorphy lifting theorems for essentially conjugate self-dual Galois representations into $\GL_n$. Existing theorems require that the residual representation have `big' image, in a certain technical sense. Our theorems are based on a strengthening of the Taylor-Wiles method which allows one to weaken this hypothesis. 
\end{abstract}
\maketitle
\tableofcontents

\section*{Introduction}

In this paper we prove new automorphy lifting theorems for essentially conjugate self-dual Galois representations into $\GL_n$, over CM imaginary fields. The main improvement on existing lifting theorems is to weaken the hypothesis `big' slightly. 

We follow the structure of the arguments of \cite{Clo08} closely. Broadly speaking, in order to prove an automorphy lifting theorem one proceeds as follows. Given a residual Galois representation, one can construct a universal deformation ring $R$ classifying deformations with certain properties (e.g. de Rham, ramified at only finitely many places). On the other hand, a space of automorphic forms gives rise to a Hecke algebra $\bbT$, and there is a map $R \rightarrow \bbT$ which classifies the `universal automorphic deformation'. One hopes to show that this is in fact an isomorphism, thereby showing that all the deformations of a given residual Galois representation of fixed type arise from automorphic forms. (This is not always possible in practice, but approximations to $R= \bbT$ type results still yield useful information).

An essential step in the proof is the introduction of auxiliary sets of primes, in what are called `Taylor-Wiles systems'. The existence of primes satisfying the relevant criteria is a problem in Galois cohomology, and a positive solution can be given when the image of the residual representation satisfies a corresponding hypothesis. 

The condition used in \cite{Clo08} was that the image of the residual representation was `big', in a certain technical sense. The main innovation in this paper is to allow more general types of ramification at the primes of our Taylor-Wiles systems, which allows us to weaken the restrictions on the image of the residual representation. We call subgroups satisfying this new condition `adequate'. 

This new condition is often satisfied in practice. In particular, we show in the appendix that, working with an $n$-dimensional representation in characteristic $l$, it is satisfied whenever $l \geq 2(n+1)$ and the representation is absolutely irreducible.

The arguments in this paper can be used to strengthen all existing automorphy lifting theorems for $\GL_n$. We have chosen to generalize two existing automorphy lifting theorems, with an eye towards applications to potential automorphy (cf. \cite{Bar10a}). First, we give a minimal lifting theorem, applicable in situations where an $R=\mathbb{T}$ type result is still inaccessible. Second, we give a strengthening of one of Geraghty's ordinary lifting theorems (cf. \cite{Ger09}) which is useful, for example, in changing the weight. 

We now give an outline of the contents of this paper. In the first section we recall some results on the Galois representations associated to automorphic forms. In section 2, we recall the definition of `big' and give the definition of `adequate' subgroups.

In section 3, we recall some foundational material from \cite{Clo08} on the deformation theory of Galois representations valued in the group $\cG_n$ (whose definition is recalled below). We also take the opportunity to define some local deformation problems that are used later. These definitions have been heavily inspired by the papers \cite{Tay08} and \cite{Bar10}.

In section 4, we define the new local deformation problem that is used in the definition of the Taylor-Wiles systems, and prove the existence of such systems under the hypothesis that the image of the residual representation is adequate. Then in section 5 we perform the local calculations on the automorphic side needed in order to be able to construct the Hecke modules in the Taylor-Wiles patching argument. This is the technical heart of the paper, and contains the only essentially new material.

In sections 6 and 7 we use these calculations to deduce an automorphy lifting theorem in the `minimal' case. Then in sections 8 and 9 we use these calculations again to extend Geraghty's ordinary lifting theorems. In section 10 we prove some technical results on the finiteness of certain deformation rings that we hope will be useful to other authors. 

Finally in the appendix we give a discussion of the properties of adequate subgroups.

\subsection*{Acknowledgements} I would like to thank my advisor Richard Taylor for drawing my attention to these problems and for many helpful conversations. I would also like to thank Florian Herzig for comments on an earlier draft of this paper.

\section*{Notation}

If $F$ is a field of characteristic zero, we write $G_F$ for its absolute Galois group. If $F/F^+$ is a quadratic extension of such fields, we write $\delta_{F/F^+}$ for the non-trivial character of $\Gal(F/F^+)$. We write $\epsilon_l : G_F \rightarrow \bbZ_l^\times$ for the $l$-adic cyclotomic character. If the prime $l$ is understood, we will write $\epsilon_l = \epsilon$.

We fix an algebraic closure $\barQl$ of $\bbQ_l$. If $F$ is a number field and $\chi$ is a character $\bbA_F^\times/F^\times \rightarrow \bbC^\times$ of type $A_0$ (i.e. the restriction of $\chi$ to $(F \otimes \bbR)^\times_0$ is given by $\prod_{\tau : F \hookrightarrow \bbC} x_\tau^{a_\tau}$ for some integers $a_\tau$), and $\iota$ is an isomorphism $\barQl \overset{\sim}{\rightarrow} \bbC$, then we write $r_{l, \iota}(\chi)$ for the associated character of $G_F \rightarrow \barQl^\times,$ given by the formula
\begin{equation*}\iota\left( (r_{l, \iota}(\chi) \circ \Art_F)(x) \prod_{\tau \in \Hom(F, \bbC)} x_{\iota^{-1}\tau}^{-a_\tau}\right) = \chi(x) \prod_{\tau\in\Hom(F, \bbC)}x_\tau^{-a_\tau},\end{equation*}
where $\Art_F$ is the global Artin map
\begin{equation*}\Art_F = \prod_v \Art_{F_v} : \bbA^\times_F \rightarrow G_F^{ab}.\end{equation*}
We normalize the local Artin maps $\Art_{F_v}$ to take uniformizers to geometric Frobenii.

If $F_v$ is a finite extension of $\bbQ_l$ inside $\barQl$, we will write $\cO_{F_v}$ for its ring of integers and $k(v)$ for its residue field. We will write $\Frob_v$ for the geometric Frobenius element in $G_{F_v}/I_{F_v}$. Suppose that $\rho : G_{F_v} \rightarrow \GL_n(\barQl)$ is a continuous representation, and take an embedding $\tau : F_v \hookrightarrow \barQl$. We write $HT_\tau(\rho)$ for the multiset of integers whose elements are the integers $i$ such that $\gr^i (\rho \otimes_{\tau, F_v} B_{dR})^{G_{F_v}} \neq 0$, with multiplicity $\dim_\barQl \gr^i (\rho \otimes_{\tau, F_v} B_{dR})^{G_{F_v}}$. Thus, for example, we have $HT_\tau(\epsilon) = \{ -1 \}$ for any $\tau$. 

If $\pi$ is an irreducible admissible representation of $\GL_n(F_v)$ over $\barQl$ then we will use the notation $r_l(\pi)$ introduced on page 81 of \cite{Clo08} to denote the $l$-adic representation associated to $\pi$ under the local Langlands correspondence, when it exists. 

If $F$ is a number field and $\rho : G_F \rightarrow \GL_n(\barQl)$ is a continuous representation, and $\tau : F \hookrightarrow \barQl$, we write $HT_\tau(\rho)$ to mean $HT_\tau(\rho|_{G_{F_v}})$, where $v$ is the place of $F$ induced by the embedding $\tau$. Thus for the character $r_{l, \iota}(\chi)$ defined above, we have
\begin{equation*}HT_\tau(r_{l, \iota}(\chi)) = \{ -a_{\iota \tau}\}.\end{equation*}

We write $\bbZ^n_+ \subset \bbZ^n$ for the set of tuples $\lambda = (\lambda_1, \dots, \lambda_n)$ of integers with $\lambda_1 \geq \dots \geq \lambda_n$.

If $\rho : G \rightarrow \GL(V)$ is a representation of a group $G$ on a vector space $V$, then we write $\ad \rho$ (resp. $\ad^0 \rho)$ for the representation of $G$ on $\End V$ given by conjugation by $\rho$ (resp. the subspace of trace zero endomorphisms).

\section{Automorphic forms on $\GL_n$ and their associated Galois representations}

This paper is dedicated to proving that certain Galois representations are automorphic. In this section we briefly review what this means. 

 Suppose that $F$ is an imaginary CM field with totally real subfield $F^+$. Let $c$ be the non-trivial element of $\Gal(F/F^+)$. Consider a pair $(\pi, \chi)$, where $\pi$ is a cuspidal automorphic representation $\pi$ of $\GL_n(\bbA_F)$ and $\chi$ a character of $\bbA_{F^+}^\times$ of type $A_0$. We say that $(\pi, \chi)$ is RAECSDC (regular, algebraic, essentially conjugate self-dual, cuspidal) if:
\begin{enumerate} \item $\pi^c \cong \pi^\vee \otimes ( \chi \circ \bbN_{F/F^+} \circ \det).$
 \item $\chi_v(-1) = (-1)^n$ for each $v \mid \infty$.
 \item The infinitesimal character of $\pi_\infty$ agrees with the infinitesimal character of an algebraic representation of $\Res^F_\bbQ \GL_n$.
\end{enumerate}
(Note that this differs slightly from the definition given in \cite{Bar09}).

Take $\lambda \in (\bbZ^n_+)^{\Hom(F,\bbC)}$. We let $\Xi_\lambda$ be the algebraic representation of $\GL_n^{\Hom(F, \bbC)}$ which is the tensor product of the irreducible representations with highest weight $\lambda_\tau$. If $\pi_\infty$ has the same infinitesimal character as $\Xi_\lambda^\vee$, then we say that $\pi$ has weight $\lambda$.

Let $C$ be an algebraically closed field of characteristic zero. If $w\in \bbZ$, We write $(\bbZ^n_+)_w^{\Hom(F, C)} \subset (\bbZ^n_+)^{\Hom(F,C)}$ for the set of $\lambda$ with
\begin{equation*}\lambda_{\tau, i} + \lambda_{\tau \circ c, n+1-i} = w\end{equation*}
for each $\tau$. Note that if $\lambda$ is the weight of a representation $\pi$ as above then $\lambda$ must lie in $(\bbZ^n_+)_w^{\Hom(F, \bbC)}$ for some $w$.

The following theorem is \cite{Bar09}, Theorem 1.2.

\begin{theorem}\label{glc} Let $(\pi, \chi)$ be as above, and choose an isomorphism $\iota : \barQl \rightarrow \bbC$. Then there exists a continuous semisimple representation $r_{l, \iota} : G_{F} \rightarrow \GL_n(\barQl)$, uniquely characterized by the following properties:
\begin{enumerate} \item $r_{l, \iota}(\pi)^c = r_{l, \iota}(\pi)^\vee \epsilon^{1-n} r_{l, \iota}(\chi)|_{G_F}.$
\item For each place $v \nmid l$ of $F$ we have
\begin{equation*}(r_{l, \iota}(\pi)|_{G_{F_v}})^{ss} \cong r_l(\iota^{-1}\pi_v)^\vee(1-n)^{ss}.\end{equation*}
\item For each place $v \mid l$ of $F$, $r_{l, \iota}(\pi)|_{G_{F_v}}$ is de Rham. If $\pi_v$ is unramified, then it is even crystalline. Moreover, we have for each $\tau : F \hookrightarrow \bbC$,
\begin{equation*}\text{HT}_{\iota^{-1} \tau}(r_{l, \iota}(\pi)) = \{ \lambda_{\tau, j} + n - j \}_{j=1, \dots, n}.\end{equation*}
\end{enumerate}
\end{theorem}
Consider a representation $\rho : G_{F} \rightarrow \GL_n(\barQl)$. Given a choice of $\lambda \in (\bbZ_+^n)^{\Hom(F, \barQl)}$, we write $\iota_\ast \lambda \in (\bbZ_+^n)^{\Hom(F, \bbC)}$ for the element with \[ (\iota_\ast \lambda)_{\tau, i} = \lambda_{\iota \tau, i}. \] If $\rho$ satisfies the conditions characterizing $r_{l, \iota}(\pi)$ above for some $\pi$ and $\iota$, with $\pi$ of weight $\iota_\ast \lambda$, then we say that $\rho$ is \emph{automorphic of weight $\lambda$}.

In Definition 5.1.2 of \cite{Ger09} it is defined what it means for an automorphic representation $\pi$ as above to be $\iota$-ordinary. If we have a representation $\rho : G_{F} \rightarrow \GL_n(\barQl)$ and an isomorphism $\rho \cong r_{l, \iota}(\pi)$ for some $\iota$ and some $\iota$-ordinary RAECSDC representation $\pi$, then we say that $\rho$ is \emph{ordinarily automorphic}.

Finally, if $L/F$ is an extension of CM imaginary fields, and $\lambda \in (\bbZ^n_+)^{\Hom(F, C)}$, then we define $\lambda_L \in (\bbZ^n_+)^{\Hom(L,C)}$ by $(\lambda_L)_\tau = \lambda_{\tau|_F}.$ We recall the following (\cite{Bar09}, Lemma 1.4).

\begin{lemma}\label{basechange} Suppose that $L/F$ is a soluble extension of CM imaginary fields. Let $\chi : G_{F^+} \rightarrow \barQl^\times$ be a continuous character with $\chi(c_v)$ independent of $v | \infty$, and let $r : G_F  \rightarrow \GL_n(\barQl)$ be a continuous semisimple representation with 
$r^c \cong r^\vee \epsilon^{1-n} \chi$.
Suppose that $r|_{G_L}$ is irreducible and automorphic of weight $\mu$. Then:
\begin{enumerate} \item There exists $\lambda \in (\bbZ^n_+)^{\Hom(F, \barQl)}$ such that $\mu = \lambda_L$.
\item $r$ is automorphic of weight $\lambda$.
\end{enumerate}
\end{lemma}

\section{Bigness revisited}\label{bigness}

Let $k$ be a finite field of characteristic $l$, and let $G$ be a subgroup of $\GL_n(k)=\GL(V)$, which acts absolutely irreducibly in the natural representation. We assume that $k$ is large enough to contain all eigenvalues of all elements of $G$. If $g \in G$ and $\alpha \in k$ is an eigenvalue of $g$, then we write $e_{g, \alpha} : V \rightarrow V$ for the $g$-equivariant projection to the generalized $\alpha$-eigenspace.

We recall the following definitions from \cite{Clo08}. 

\begin{definition} $\cG_n$ is a group scheme over $\bbZ$, defined as the semi-direct product of $\GL_n \times \GL_1$ by the group $\{1, j\}$, which acts on $\GL_n \times \GL_1$ by
\begin{equation*}j(g, \mu)j^{-1} = (\mu {}^t g^{-1}, \mu).\end{equation*}
We write $\nu$ for the natural homomorphism $\cG_n \rightarrow \GL_1$ which takes $(g, \mu)$ to $\mu$ and $j$ to $-1$, and $\cG_n^0$ for the connected component of $\cG_n$. Note that $\cG_n$ acts naturally on $\Lie \GL_n \subset \Lie \cG_n$. We will sometimes write this representation as $\ad V$.
\end{definition} 

\begin{definition} Let $G$ be as above. We say that $G$ is \emph{big} if the following conditions are satisfied: \begin{itemize} \item $H^0(G, \ad^0 V)=0.$
\item $H^1(G, k) = 0$.
\item $H^1(G, \ad^0 V)=0$.
\item For every irreducible $k[G]$-submodule $W \subset \ad^0 V$, there exists an element $g \in G$ with a multiplicity one eigenvalue $\alpha$ such that $\tr e_{g, \alpha} W \neq 0$.
\end{itemize}

Similarly, we say that a subgroup $G$ of $\cG_n(k)$ is \emph{big} if it satisfies the following conditions. Let $G^0 = G \cap \cG_n^0(k)$. \begin{itemize}
\item $H^0(G, \ad V)=0$.
\item $H^1(G^0, k)=0$.
\item $H^1(G, \ad V)=0$.
\item For every irreducible $k[G]$-submodule $W \subset \ad V$, there exists an element $g \in G^0$ with a multiplicity one eigenvalue $\alpha$ such that $\tr e_{g, \alpha} W \neq 0$. 
\end{itemize}
\end{definition}

We remark as in \cite{Clo08} that if $G/G^0$ surjects onto $\cG_n(k)/\cG_n^0(k)$ and $G^0$ is big then $G$ is big. We now make the following revised definitions:

\begin{definition} Let $G \subset \GL(V)$ be as above. We say that $G$ is \emph{adequate} if the following conditions are satisfied.
 \begin{itemize} \item $H^0(G, \ad^0 V)=0.$
\item $H^1(G, k) = 0$.
\item $H^1(G, \ad^0 V)=0$.
\item For every irreducible $k[G]$-submodule $W \subset \ad^0 V$, there exists an element $g \in G$ with an eigenvalue $\alpha$ such that $\tr e_{g, \alpha} W \neq 0$.
\end{itemize}

Similarly, we say that a subgroup $G$ of $\cG_n(k)$ is \emph{adequate} if it satisfies the following conditions. Let $G^0 = G \cap \cG_n^0(k)$. \begin{itemize}
\item $H^0(G, \ad V)=0$.
\item $H^1(G^0, k)=0$.
\item $H^1(G, \ad V)=0$.
\item For every irreducible $k[G]$-submodule $W \subset \ad V$, there exists an element $g \in G^0$ with an eigenvalue $\alpha$ such that $\tr e_{g, \alpha} W \neq 0$. 
\end{itemize}
\end{definition}
Thus the only difference between `adequate' and `big' is that we no longer require the eigenvalue $\alpha$ to have multiplicity one. We have the following.
\begin{lemma} \begin{enumerate} \item Big $\Rightarrow$ adequate.
\item Suppose that $l \geq 2(n+1)$. If $G$ is a subgroup of $\cG_n(k)$ which surjects onto $\cG_n(k)/\cG_n^0(k)$, and $G^0$ acts absolutely irreducibly, then $G$ is adequate.
\end{enumerate}
\end{lemma}
\begin{proof}
The first part is obvious. For the second, it is proved in the appendix that when $l \geq 2(n+1)$, any subgroup of $\GL_n(k)$ which acts absolutely irreducibly is adequate. The result now follows from the remarks above.
\end{proof}

\section{Deformations of Galois representations}\label{galoisdef}

In this section we recall from \cite{Clo08} some of the concepts that we will need relating to deformations of Galois representations valued in $\cG_n$. Let $F$ be an imaginary CM field with totally real subfield $F^+$. We fix a finite set of places $S$ of $F^+$ which split in $F$ and write $F(S)$ for the maximal extension of $F$ unramified outside $S$. We write $G_{F^+, S}=\Gal(F(S)/F^+)$ and $G_{F, S} \subset G_{F^+, S}$ for the subset of elements fixing $F$. For each $v \in S$ we choose a place $\wv$ of $F$ above it, and write $\widetilde{S}$ for the set of these places. 

We fix a finite field $k$ of characteristic $l$ and a representation $\overline{r} : G_{F^+, S} \rightarrow \cG_n(k)$ such that $G_{F,S} = \overline{r}^{-1}(\GL_n \times \GL_1(k))$. Let $K$ be a finite extension of $\bbQ_l$ in $\barQl$ with ring of integers $\cO$, maximal ideal $\lambda$, and residue field $k$. Choose a character $\chi : G_{F^+, S} \rightarrow \cO^\times$ such that $\nu \circ \overline{r} = \overline{\chi}$. We will consider deformations of $\overline{r}$ to objects of $\cC_\cO$, the category of complete Noetherian local $\cO$-algebras with residue field $k$. If $\wv \in S$, we write $\overline{r}|_{G_{F_\wv}}$ for the composite
\begin{equation*}G_{F_\wv} \rightarrow G_{F, S} \rightarrow \cG_n^0(k) \rightarrow \GL_n(k).\end{equation*}

\begin{definition} A lifting of $\overline{r}$ (resp. $\overline{r}|_{G_{F_\wv}}$) to an object $R$ of $\cC_\cO$ is a continuous homomorphism $r : G_{F^+, S} \rightarrow \cG_n(R)$ (resp. $r : G_{F_\wv} \rightarrow \GL_n(R))$ with $r \mod \ffrm_R = \overline{r}$ (resp. = $\overline{r}|_{G_{F_\wv}})$ and $\nu \circ r = \chi$ (resp. no further condition). Two liftings are said to be \emph{equivalent} if they are conjugate by an element of $1 + M_n(\ffrm_R) \subset \GL_n(R)$. An equivalence class of liftings is called a \emph{deformation}.

Let $T \subset S$. By a \emph{$T$-framed lifting} of $\overline{r}$ to $R$ we mean a tuple $(r; \alpha_v)_{v \in T}$ where $r$ is a lifting of $\overline{r}$ and $\alpha_v \in 1 + M_n(\ffrm_R)$. We call two framed liftings $(r; \alpha_v)$ and $(r'; \alpha'_v)$ equivalent if there is an element $\beta \in 1 + M_n(\ffrm_R)$ with $r' = \beta r \beta^{-1}$ and $\alpha'_v = \beta \alpha_v$. By a \emph{$T$-framed deformation} of $\overline{r}$ we mean an equivalence class of framed liftings.
\end{definition}

\begin{definition} If $v \in S$ then we define a \emph{local deformation problem} at $v$ to be a subfunctor $\cD_v$ of the functor of all liftings of $\overline{r}|_{G_{F_\wv}}$ to objects of $\cC_\cO$ satisfying the following conditions: 
\begin{enumerate} \item $(k, \overline{r}) \in \cD_v.$
\item Suppose that $(R_1, r_1)$ and $(R_2, r_2) \in \cD_v$, that $I_1$ (resp. $(I_2)$) is a closed ideal of $R_1$ (resp. $R_2$) and that $f : R_1/I_1 \rightarrow R_2/I_2$ is an isomorphism in $\cC_\cO$ such that $f(r_1 \mod I_1) = r_2 \mod I_2$. Let $R_3$ denote the subring of $R_1 \times R_2$ consisting of pairs with the same image in $R_1/I_1 \cong R_2/I_2$. Then $(R_3, r_1 \times r_2) \in \cD_v$.
\item If $(R_j, r_j)$ is an inverse system of elements of $\cD_v$ then \begin{equation*}(\lim R_j, \lim r_j) \in \cD_v.\end{equation*}
\item $\cD_v$ is closed under equivalence.
\item If $R \subset S$ is an inclusion in $\cC_\cO$ and if $r : G_F \rightarrow \GL_n(R)$ is a lifting of $\overline{r}$ such that $(S, r) \in \cD_v$ then $(R, r) \in \cD_v$.
\end{enumerate}
\end{definition}

We recall (cf. \cite{Bar09}, Lemma 3.2) that to give a local deformation problem, it suffices to give a quotient $R$ of the universal lifting ring $R^\square$ of $\overline{r}|_{G_{F_\wv}}$ which is  reduced, and such that the defining ideal of $R$ is invariant under the natural conjugation action of $1 + M_n(\ffrm_{R^\square})$.

Given a collection of deformation problems $\cD_v$ for $v \in S$, we have a (global) deformation problem
\begin{equation*}\cS = \left(F/F^+, S, \widetilde{S}, \cO, \overline{r}, \chi, \{ \cD_v\}_{v \in S}\right).\end{equation*}
\begin{definition} Let $T \subset S$. We call a $T$-framed lifting $(r; \alpha_v)_{v \in T}$ of $\overline{r}$ of \emph{type $\cS$} if for all $v \in S$, the restriction $r|_{G_{F_\wv}}$ lies in $\cD_v$. We say that a $T$-framed deformation is of type $\cS$ is some (equivalently any) element of the equivalence class is of type $\cS$.

We let $\Def_\cS^{\square_T}$ denote the functor which associates to an object $R$ of $\cC_\cO$ the set of all $T$-framed deformations of $\overline{r}$ to $R$ of type $\cS$. If $T = S$ then we refer to framed deformations and write $\Def_\cS^\square$. If $T=\emptyset$ we refer to deformations and write $\Def_\cS$.

If $R_v$ denotes the ring representing the local deformation problem $\cD_v$, then we write
\begin{equation*} R^{\text{loc}}_{\cS,T} = \widehat{\otimes}_{v \in T} R_v.
\end{equation*}

\end{definition}
The following is \cite{Clo08}, Proposition 2.2.9.
\begin{proposition} Suppose that $\overline{r}|_{G_{F,S}}$ is absolutely irreducible. Then the functors $\Def_\cS^{\square_T}, \Def_\cS^{\square}, \Def_\cS$ are represented by objects of $\cC_\cO$. We write respectively $R_\cS^{\square_T}, R_\cS^\square$ and $R_\cS^\text{univ}$ for the representing objects.
\end{proposition}

\subsection*{Local deformation problems in the case $l = p$}

Here we define some local deformation problems that will be useful later. Fix a finite extension $L_v$ of $\bbQ_l$. Assume that $K$ is large enough to contain every embedding of $L_v$ in $\barQl$. Now fix a representation $\overline{\rho} : G_{L_v} \rightarrow \GL_n(k)$, and write $\rho^\square : G_{L_v} \rightarrow \GL_n(R^\square)$ for the universal (unrestricted) lifting. Given an element $\lambda \in (\bbZ_+^n)^{\Hom(L_v, K)}$ we define for every $\tau \in \Hom(L_v, K)$ a multiset of integers 
\begin{equation*}H_\tau = \{ \lambda_{\tau, j} + n - j\}_{j=1, \dots, n}.\end{equation*}

\begin{theorem} Let $\lambda$ be as above. There exists a reduced $l$-torsion free quotient $R_v^{\lambda, \text{cr}}$ of $R^\square$, uniquely characterised by the property that a homomorphism $\zeta : R^\square \rightarrow \barQl$ factors through $R_v^{\lambda, \text{cr}}$ if and only if $\rho = \zeta \circ \rho^\square : G_{L_v} \rightarrow \GL_n(\barQl)$ is crystalline, with $HT_\tau(\rho) = H_\tau$ for each $\tau : L_v \hookrightarrow \barQl$.

Moreover, $\Spec R_v^{\lambda, \text{cr}}[1/l]$ is formally smooth over $K$ and equidimensional of dimension $n^2 + \frac{1}{2}n(n-1)[L_v : \bbQ_l]$.
\end{theorem}
\begin{proof} This is deduced in \cite{Bar10} from the theorems of \cite{Kis08}.
\end{proof}
Note in particular that the ring $R_v^{\lambda, \text{cr}}$ defines a local deformation problem. There is also a potentially crystalline version:
\begin{theorem} Let $\lambda$ be as above, and let $L_v'$ be a finite extension of $L_v$. There exists a reduced $l$-torsion free quotient $R_v^{\lambda, L_v'-\text{cr}}$ of $R^\square$, uniquely characterised by the property that a homomorphism $\zeta : R^\square \rightarrow \barQl$ factors through $R_v^{\lambda, L_v'-\text{cr}}$ if and only if $\rho = \zeta \circ \rho^\square : G_{L_v} \rightarrow \GL_n(\barQl)$ is de Rham with $HT_\tau(\rho) = H_\tau$ for each $\tau : L_v \hookrightarrow \barQl$, and moreover $\rho|_{G_{L'_v}}$ is crystalline.

Moreover, $\Spec R_v^{\lambda, L_v'-\text{cr}}[1/l]$ is equidimensional of dimension $n^2 + \frac{1}{2}n(n-1)[L_v : \bbQ_l]$.
\end{theorem}
\begin{proof}
This also follows in a straightforward manner from the results of \cite{Kis08}.
\end{proof}
\begin{definition} Let $\rho_1, \rho_2 : G_{L_v} \rightarrow \GL_n(\cO)$ be continuous lifts of ${\overline{\rho}}$. We write $\rho_1 \sim \rho_2$ if the following are satisfied:
\begin{enumerate} \item There exists $\lambda \in (\bbZ^n_+)^{\Hom(L_v,K)}$ and a finite extension $L'_v/L_v$ such that both $\rho_1$ and $\rho_2$ correspond to points of $R_v^{\lambda, L'_v-\text{cr}}$.
\item $\rho_1$ and $\rho_2$ give rise to closed points on the same irreducible component of $\Spec R_v^{\lambda,  L'_v-\text{cr}} \otimes_\cO \barQl$.
\end{enumerate}
\end{definition}
Note that $\rho_1 \sim \rho_2$ implies that $\rho_1|_{G_{L''_v}} \sim \rho_2|_{G_{L''_v}}$ for any finite extension $L''_v/L_v$. 

Suppose that $\cC$ is an irreducible component of $R_v^{\lambda, L'_v-\text{cr}} \otimes_\cO \barQl$. Then we write $R_v^{\lambda, \cC}$ for the maximal reduced, $l$-torsion free quotient of $R_v^{\lambda, L'_v-\text{cr}}$ such that $\Spec R_v^{\lambda, \cC} \otimes_\cO \barQl$ is equal to this irreducible component. (This exists provided that $K$ is large enough, which we always assume).

\begin{lemma}\label{defproblem} Say that a lift $(\rho, R)$ is of type $\cD_v^\cC$ if the induced map $R^\square \rightarrow R$ factors through $R_v^{\lambda, \cC}$. Then $\cD_v^\cC$ is a local deformation problem.
\end{lemma}
\begin{proof}
The proof is identical to that of Lemma 1.2.1 of \cite{Bar10a}.
\end{proof}

\begin{definition}\label{ord} Let $\lambda \in (\bbZ_+^n)^{\Hom(L_v, K)}$. We say that a continuous representation $\rho : G_{L_v} \rightarrow \GL_n(\cO)$ is \emph{ordinary of weight $\lambda$} if:
\begin{enumerate} \item There exists a increasing invariant filtration $\Fil^i$ of $\cO^n$, with each  $\gr^i \cO^n$ an $\cO$-module of rank one. Write $\chi_i$ for the character $G_{L_v} \rightarrow \cO^\times$ giving the action on $\gr^i \cO^n$.
\item For every $\alpha \in L_v^\times$ sufficiently close to 1, we have
\begin{equation*}(\chi_i \circ \Art_{L_v}(\alpha)) = \prod_\tau (\tau(\alpha))^{-(\lambda_{\tau, n - i + 1} + i - 1)}.\end{equation*}
\end{enumerate}
We say that $\rho$ is \emph{ordinary} if it is ordinary of some weight.
\end{definition}
Thus $HT_\tau(\rho) = H_\tau$, whenever $\rho$ is ordinary of weight $\lambda$.
\begin{theorem} Suppose $\overline{\rho}$ is the trivial representation. Write $\Lambda_v$ for the completed group ring of the group $I_{L_v}^{ab}(l)^n$, where $(l)$ denotes pro-$l$ completion. There exists a reduced $l$-torsion free quotient $R^{\triangle, ar}_{\Lambda_v}$ of $\Lambda_v \widehat{\otimes}_\cO R^\square$ satisfying the following properties:
\begin{enumerate} \item Let $\widetilde{\chi} = (\widetilde{\chi}_1, \dots, \widetilde{\chi}_n)$ denote the universal $n$-tuple of characters of $I_{L_v}^{ab}(l)$. Let $\zeta : R^{\triangle, ar}_{\Lambda_v} \rightarrow \barQl$ be a continuous homomorphism. Then $\zeta \circ \rho^\square$ is conjugate to a representation
\begin{equation*}\left(\begin{array}{ccccc} \chi_1 & \ast & \dots & \ast & \ast \\	
0 & \chi_2 & \dots & \ast & \ast \\
\vdots & \vdots & \ddots & \vdots & \vdots \\
0 & 0 & \dots & \chi_{n-1} & \ast \\
0 & 0 & \dots & 0 & \chi_n
\end{array}\right)\end{equation*}
where $\chi|_{I_{L_v}} = (\chi_1|_{I_{L_v}}, \dots, \chi_n|_{I_{L_v}}) = \zeta \circ \widetilde{\chi}.$
\item Let $Q$ be a minimal prime of $\Lambda$. Then $R^{\triangle, ar}_{\Lambda_v}/Q$ is irreducible.
\end{enumerate}
Moreover, $R^{\triangle, ar}_{\Lambda_v}$ defines a local deformation problem.
\end{theorem}
\begin{proof}
This is all proved in \cite{Ger09}.
\end{proof}

Finally we define a lifting ring which classifies representations which are both ordinary and semistable. 
\begin{theorem} There exists a reduced $l$-torsion free quotient $R_v^{\lambda, \text{ss-ord}}$ of $R^\square$ such that if $\zeta : R^\square \rightarrow \barQl$ is a homomorphism, then $\zeta$ factors through $R_v^{\lambda, \text{ss-ord}}$ if and only if: 
\begin{enumerate} \item $\zeta \circ \rho^\square$ is ordinary of weight $\lambda$.
\item $\zeta \circ \rho^\square$ is semistable.
\end{enumerate}
Moreover, $\Spec R_v^{\lambda, \text{ss-ord}}$ is equidimensional of dimension $n^2 + \frac{1}{2}n(n-1)[L_v : \bbQ_l]$.
\end{theorem}
\begin{proof}
This follows from the results of \cite{Ger09}, Section 3  (where the ring $R_v^{\lambda, \text{ss-ord}}$ is denoted $R^{\triangle_{\lambda_w}, \text{st}}).$
\end{proof}
\subsection*{Local deformation problems in the case $l \neq p$}

Let $L_v$ be a finite extension of $\bbQ_p$. Let $\overline{\rho} : G_{L_v} \rightarrow \GL_n(\cO)$ be a continuous representation. Let $R^\square$ denote the universal (unrestricted) lifting ring. Then $R^\square\otimes_\cO\barQl$ is equidimensional of dimension $n^2$ (see \cite{Gee06}). Let $\cC$ be an irreducible component or a set of irreducible components. Then we will write $R_v^\cC$ for the maximal quotient of $R^\square$ which is reduced and $l$-torsion free, and such that $\Spec R_v^\cC \otimes_\cO \barQl$ consists of the irreducible components in $\cC$. (Again, this requires $K$ to be sufficiently large). We write $R_v^{\overline{\square}}$ for the maximal reduced and $l$-torsion free quotient of $R^\square$.

The proof of the following lemma is exactly the same as that of Lemma \ref{defproblem} above.
\begin{lemma} Say that a lift $(\rho, R)$ is of type $\cD^\cC_v$ if the induced map $R^\square \rightarrow R$ factors through $R_v^\cC$. Then $\cD^\cC_v$ is a local deformation problem.
\end{lemma}

\begin{definition} Let $\rho_1, \rho_2 : G_{L_v} \rightarrow \GL_n(\cO)$ be continuous lifts of ${\overline{\rho}}$. We write $\rho_1 \sim \rho_2$ if the points of $\Spec R^\square \otimes_\cO \barQl$ induced by $\rho_1$ and $\rho_2$ lie on a common irreducible component.

We write $\rho_1 \leadsto \rho_2$ if $\rho_1 \sim \rho_2$ and moreover the point of $\Spec R^\square \otimes_\cO \barQl$ corresponding to $\rho_1$ lies on a \emph{unique} irreducible component.
\end{definition}

\begin{proposition} Suppose that $\rho_1, \rho_2 : G_{L_v} \rightarrow \GL_n(\cO)$ are continuous unramified lifts of $\overline{\rho}$. Suppose that $(\rho_1 \otimes_\cO \barQl)^{ss} = (r_l(\pi)^\vee(1-n))^{ss}$ for some generic unramified smooth irreducible representation $\pi$ of $\GL_n(L_v)$ over $\barQl$. Then $\rho_1 \leadsto \rho_2$.
\end{proposition}
\begin{proof} It is immediate that $\rho_1 \sim \rho_2$. A standard argument shows that $\Spec R^\square \otimes_\cO \barQl$ is formally smooth at the point corresponding to $\rho_1 \otimes_\cO \barQl$ if the group $H^0(G_{L_v}, \ad ( \rho_1 \otimes_\cO \barQl ) (1))$ vanishes. The genericity of $\pi$ implies that no two of the eigenvalues of $\rho_1(\Frob_v)$ differ by $\# k(v) = q_v$, so the result follows.
\end{proof}

We will need a slight generalization of the `level-raising' deformation problems of \cite{Tay08}. Before defining these, we give a geometrical lemma.

\begin{lemma}\label{matrices} Let $q$ be an integer, and let $\alpha_1, \dots, \alpha_n$ be roots of unity in $\cO^\times$. Let $\cM(\prod_{i=1}^n(X-\alpha_i),q)$ be the moduli space over $\cO$ of pairs of $n \times n$ matrices $(\Phi, \Sigma)$ where $\Phi$ is invertible, $\cha_\Sigma(X) = \prod_{i=1}^n(X-\alpha_i)$, and $\Phi \Sigma \Phi^{-1} = \Sigma^q.$ (Thus  $\cM(\prod_{i=1}^n(X-\alpha_i),q)$ is an affine $\cO$-scheme).

\begin{enumerate} \item Suppose first that $\alpha_j = 1$ for each $j$. Let $\cM_i$ denote the irreducible components of $\cM((X-1)^n, q)$ with their reduced subscheme structure. Then each $\cM_i \otimes K$ is non-empty of dimension $n^2$. Moreover, the distinct irreducible components of $\cM((X-1)^n,q) \otimes k$ are the $\cM_i \otimes k$ and each is non-empty and generically reduced.
\item Suppose that $q \equiv 1 \mod l$, and that the $\alpha_j$ are distinct roots of unity in $1 + \lambda \subset \cO^\times$. If $q \neq 1$ we suppose that they are distinct $(q-1)^\text{st}$ roots of unity. Then $\cM(\prod_{i=1}^n(X-\alpha_i),q)^\text{red}$ is flat over $\cO$.
\end{enumerate}
\end{lemma}
\begin{proof}

We treat the first part of the lemma.  When $l>n$ and $q \equiv 1 \mod l$, this is Lemma 3.2 of \cite{Tay08}. In fact, the proof there works without any restriction on $q$. Taylor has given another proof valid without restriction on $l$, and I thank him for allowing me to reproduce this here.

Let us abbreviate $\cM = \cM((X-1)^n, q)$. Let $\Nil$ be the space of $n \times n$ nilpotent matrices, and $\Pol$ the space of monic degree $n$ polynomials. $\cM$ has natural maps to $\Nil$ and $\Pol$ taking a pair $(\Phi, \Sigma)$ to $N = \Sigma - 1$ and $\cha_\Phi(X)$, respectively.

Let $\sigma$ be a partition of $n$, say $n = n_1 + \dots + n_r$, and let $\Nil(\sigma)$ be the corresponding subvariety of $\Nil$, as defined in \cite{Tay08} (it is the variety of nilpotent matrices whose Jordan normal form refines the partition $\sigma$). It has an open subvariey $\Nil(\sigma)^0$ (the variety of nilpotent matrices whose Jordan normal form corresponds to the partition $\sigma$). Let $\Pol(\sigma, q)$ be the subvariety of $\Pol$ consisting of polynomials whose multiset of roots can be partitioned into $r$ multisets of the form $\{ \alpha, q\alpha, \dots, q^{n_i - 1}\alpha \}$.

We consider the following subschemes of $\cM$:
\begin{enumerate} \item $\cM(\sigma)^0$ is the locally closed pre-image of $\Nil(\sigma)^0$ in $\cM$.
\item $\cM(\sigma)$ is the reduced subscheme of the closure of $\cM(\sigma)^0$ in $\cM$.
\item $\cM(\sigma)'$ is the reduced subscheme of the intersection of the pre-image of $\Nil(\sigma)$ under the first map and the pre-image of $\Pol(\sigma, q)$ under the second map above.
\end{enumerate}

Let $L$ be a field, and let $(\Phi, \Sigma)$ be an $L$-point of $\cM(\sigma)^0$. Then $\cha_\Phi(X) \in \Pol(\sigma, q)(L)$. When $L$ has characteristic zero, this can be deduced just as in \cite{Tay08}. Suppose instead that $L$ has characteristic $l$. Let $M$ be an integer such that $l^M > n$, and choose a positive integer $m$ such that $mq \equiv 1 \mod l^M$. Then $\Phi \Sigma \Phi^{-1} = \Sigma^{mq} = \Sigma$. Writing $N = \Sigma - 1$, we have 
\begin{equation*}\Phi(N+1)^m\Phi^{-1} = (N+1)\end{equation*}
and hence
\begin{equation*}\Phi(m^i N^i + O(N^{i+1})) = N^i \Phi.\end{equation*}
It follows that $\Phi$ preserves $\ker N^i$. We deduce that
\begin{equation*}\cha_{\Phi|_{\ker N^{i+1}/\ker N^{i}}}(X) \mid \cha_{\Phi|_{\ker N^{i}/\ker N^{i-1}}}(qX),\end{equation*}
and hence $\cha_\Phi(X) \in \Pol(\sigma, q)(L).$

We deduce the chain of inclusions
\begin{equation*}(\cM(\sigma)^0)^\text{red} \subset \cM(\sigma) \subset \cM(\sigma)',\end{equation*}
these schemes being reduced and the inclusions holding on field-valued points. Just as in \cite{Tay08}, we have that the projection $\cM(\sigma)^0 \rightarrow \Nil(\sigma)^0$ is isomorphic locally in the Zariski topology to the projection
\begin{equation*}\Nil(\sigma)^0 \times Z_{\GL_n}(N(\sigma)) \rightarrow \Nil(\sigma)^0,\end{equation*}
where $N(\sigma)$ is a nilpotent Jordan matrix corresponding to the partition $\sigma$. In particular, $\cM(\sigma)^0$ is smooth over $\cO$ and the fibres are integral of dimension $n^2$. Now, choose $a_1, \dots, a_r \in \overline{k}^\times$ such that the multisets $S_i = \{ a_i, q a_i, \dots, q^{n_i - 1}a_i\}$ for $i = 1, \dots, r$ are disjoint. Choose a matrix $\Phi(\sigma, a)$ such that \begin{equation*}\cha_\Phi(X) = \prod_{j=1}^r\prod_{\alpha \in S_j} (X - \alpha)\end{equation*}
and $(\Phi, 1 + N(\sigma)) \in \cM(\sigma)^0$. Then this tuple defines a point of
\begin{equation*}\left( \cM(\sigma)^0 - \bigcup_{\sigma' \neq \sigma} \cM(\sigma')' \right)\left(\overline{k}\right) \subset \left( \cM(\sigma) - \bigcup_{\sigma' \neq \sigma} \cM(\sigma') \right)\left(\overline{k}\right) .\end{equation*}
Thus the $\cM(\sigma)$ are the irreducible components of $\cM$, and they satisfy the conditions of the lemma.

The second part of the lemma follows from an analogous modification to the argument given in \cite{Tay08} (replace $\exp(N)$ by the map $N \mapsto 1 + N$; the proof then goes through essentially unchanged).
\end{proof}

Now let $\chi_{v,1}, \dots, \chi_{v,n} : \cO_{L_v}^\times \rightarrow \cO^\times$ be finite order characters, which become trivial on reduction modulo $\lambda$. Suppose that $\#k(v) = q_v \equiv 1 \mod l$. Suppose that $\overline{\rho}$ is trivial. We write $\cD_v^\chi$ for the set of liftings $\rho$ of $\overline{\rho}$ to objects of $\cO$ such that for all $\sigma \in I_{L_v}$, we have
\begin{equation*}\cha_{\rho(\sigma)}(X) = \prod_{i=1}^n\left(X-\chi_{v,j}(\Art_{L_v}(\sigma)\right)^{-1}).\end{equation*}
This is a local deformation problem, studied in \cite{Tay08}. Write $R_v^{\chi_v}$ for the corresponding local lifting ring. We have the following.
\begin{proposition}\label{levelraising} \begin{enumerate}\item Suppose that $\chi_{v,j} = 1$ for each $j$. Then each irreducible component of $R_v^1$ has dimension $n^2 + 1$, and every prime of $R_v^1$ minimal over $\lambda$ contains a unique minimal prime. Every generic point is of characteristic zero.
 \item Suppose that the $\chi_{v,j}$ are pairwise distinct. Then $\Spec R_v^{\chi_v}$ is irreducible of dimension $n^2 + 1$, and its generic point is of characteristic zero.
\end{enumerate}
\end{proposition}
\begin{proof} The first part of the proposition follows from Lemma 2.7 of \cite{Tay08} and Lemma \ref{matrices} above. The second part follows from Lemmas 2.7, 3.3 and 3.4 of \cite{Tay08} and Lemma \ref{matrices} above. (The proofs of these other results do not use the assumption $l > n$).
\end{proof}

Now suppose that $\overline{\rho}$ is not necessarily trivial, and that $q_v \not\equiv 1 \mod l$. We consider the set $\cD^a_v$ of liftings $\rho$ to objects of $\cO$ such that for all $\sigma \in I_{L_v}$, we have
\begin{equation*}\cha_{\rho(\sigma)}(X) = (X-1)^n.\end{equation*}
This is a local deformation problem. Let $R_v^a$ denote the corresponding local lifting ring.

\begin{proposition}\label{smallprimes} Each irreducible component of $R_v^a$ has dimension $n^2 + 1$, and every prime of $R_v^a$ minimal over $\lambda$ contains a unique minimal prime. Every generic point is of characteristic zero.
\end{proposition}
\begin{proof}
This follows from Lemma 2.7 of \cite{Tay08} and Lemma \ref{matrices} above (complete at the point in the special fibre of $\cM$ corresponding to $\overline{\rho}$).
\end{proof}

Finally, suppose that $\overline{\rho}$ is unramified and make no hypothesis on $q_v$. We write $\cD^{ur}_v$ for the collection of all unramified liftings of $\overline{\rho}$.

\begin{proposition} $\cD_v^{ur}$ is a local deformation problem, and the corresponding lifting ring $R^{ur}_v$ is formally smooth over $\cO$.
\end{proposition}

\section{Galois theory}
The following generalizes the discussion in Section 2.5 of \cite{Clo08}. 

\begin{definition} Suppose that \begin{equation*}\cS = \left(F/F^+, S, \widetilde{S}, \cO, \overline{r}, \chi, \{ \cD_v \}_{v \in S}\right)\end{equation*} is a global deformation problem and that $T \subset S$. Let $Q$ be a finite set of primes $v \not\in S$ of $F^+$ which split in $F$ and for which \begin{equation*}\bbN v \equiv 1 \mod l.\end{equation*} Let $\widetilde{Q}$ denote a set of primes of $F$ containing exactly one prime $\wv$ of $F$ lying above each prime of $Q$. 

If $v \in Q$ then $\overline{r}|_{G_{F_\wv}}$ is unramified. We write it in the form $\overline{s}_v \oplus \overline{\psi}_v$ where $\psi_v$ is an eigenspace of Frobenius corresponding to an eigenvalue $\alpha_v$, on which Frobenius acts semisimply. Then we define a second global deformation problem \begin{equation*}\cS_Q = \cS_{Q, \{ \overline{\psi}_v\}_{v \in Q}} = (F/F^+, S \cup Q, \widetilde{S} \cup \widetilde{Q}, \cO, \overline{r}, \chi, \{ \cD_v \}_{v \in S \cup Q}),\end{equation*} where for $v \in Q$ we take $\cD_v$ to consist of all lifts which are $1 + M_n(\mathfrak{m}_R)$-conjugate to one of the form $s_v \oplus \psi_v$, where $s_v$ is unramified and $\psi_v$ may be ramified, but the image of inertia under $\psi_v$ is contained in the set of scalar matrices.
\end{definition}
Note that if one chooses the $\overline{\psi}_v$ to be one-dimensional, then this just recovers Definition 2.5.7 of \cite{Clo08}. The following lemma shows that the above definition makes sense.
\begin{lemma} Let $F_v$ be a finite extension of $\bbQ_p$, and let $\overline{r} : G_{F_v} \rightarrow \GL_n(k) = \overline{s} \oplus \overline{\psi}$ and $\cD_v$ be as above. Then $\cD_v$ is a local deformation problem.
\end{lemma}
\begin{proof} Of the conditions required for $\cD_v$ to be a local deformation problem (recalled in the previous section), only (ii) and (iii) present any difficulty. Let us treat (ii) first. Let $F$ be a Frobenius lift. By Hensel's lemma, applied to the characteristic polynomial of $r_i(F)$, there is for each $i$ a unique splitting $r_i \cong s_i \oplus \psi_i$ where $s_i$, $\psi_i$ are subrepresentations of $r_i$ lifting $\overline{s}$ and $\overline{\psi}$, respectively. Similarly, we have $r_i \mod I_i \cong \widetilde{s} \oplus \widetilde{\psi}$ for unique subrepresentations $\widetilde{s}, \widetilde{\psi}$. 

Let $\widetilde{e}_1, \dots, \widetilde{e}_n$ be elements of $R_i/I_i$ lifting the standard basis of $k^n$ such that $\widetilde{e}_1, \dots, \widetilde{e}_m$ lie in $\widetilde{s}$ and $\widetilde{e}_{m+1}, \dots, \widetilde{e}_{n}$ lie in $\widetilde{\psi}$. By Nakayama's lemma, $\widetilde{e}_1, \dots, \widetilde{e}_m$ is a basis of $\widetilde{s}$, and similarly for $\widetilde{\psi}$.

Now lift these again to elements $e^i_1, \dots, e^i_n$ of $R_i$ for $i=1,2$. Let $A_i$ be the change of basis matrix. Then we have $A_i \in 1 + M_n(\ffrm_{R_i})$. Setting $A_3 = A_1 \times A_2$, we find $A_3 \in 1 + M_n(\ffrm_{R_3})$ and 
\begin{equation*}A_3 (r_1 \times r_2) A_3^{-1} = (s_1 \times s_2) \oplus (\psi_1 \times \psi_2).\end{equation*}
Hence $(R_3, r_1 \times r_2) \in \cD_v$.

For (iii), let $(R, r) = (\lim R_j, \lim r_j)$. Then applying Hensel's lemma again gives a splitting $r \cong s \oplus \psi$, for subrepresentations $s, \psi$. (Note that the property of inertia commuting with every choice of Frobenius lift $F$ is preserved under passing to the limit). Applying the uniqueness of the splitting of each $r_j$, we see that $s$ must be unramified, and that inertia acts centrally through $\psi$. The result follows.
\end{proof}
The following is proved in the same way as Lemma 2.5.8 of \cite{Clo08} (in fact the statement below is identical, but our notation $\cS_Q$ now means something slightly more general). We refer back to \cite{Clo08}, Section 2 for the definition of $L_v$ and $L_v^\perp$ below (informally, these are the local conditions in cohomology coming from our choice of local deformation problems $\cD_v$).
\begin{lemma} Suppose that we are in the situation of the above definition, and that $\overline{r}$ is absolutely irreducible. Suppose also that for $v \in S - T$ we have
\begin{equation*}\dim_k L_v - \dim_k H^0(G_{F_\wv}, \ad \overline {r}) = \left\{\begin{array}{ll} [F_v^+ : \bbQ_l]n(n-1)/2 & \text{if }v \mid l \\ 0 & \text{if } v \nmid l.\end{array}\right.\end{equation*} Then $R^{\square_T}_{\cS_Q}$ can be topologically generated over $R^\text{loc}_{\cS, T} = R^\text{loc}_{\cS_Q, T}$ by \begin{equation*}\dim_k H^1_{\cL(Q)^\perp, T}(G_{F^+, S}, \ad \overline{r}(1)) + \# Q - \sum_{v \in T, v \mid l} [F_v^+ : \bbQ_l]n(n-1)/2\end{equation*} \begin{equation*} - \dim_k H^0(G_{F^+, S}, \ad \overline{r}(1)) - n \sum_{v \mid \infty} (1 + \chi(c_v))/2\end{equation*} elements.
\end{lemma}

\begin{proposition}\label{Chebotarev} Let $q_0 \in \bbZ_{\geq 0}$. Suppose given a deformation problem \begin{equation*}\cS = (F/F^+, S, \widetilde{S}, \cO, \overline{r}, \chi, \{ \cD_v \}_{v \in S}).\end{equation*} Suppose that $\overline{r}$ is absolutely irreducible and that $\overline{r}(G_{F^+(\zeta_l)})$ is adequate. Suppose also that for $v \in S - T$ we have \begin{equation*}\dim_k L_v - \dim_k H^0(G_{F_\wv}, \ad \overline {r}) = \left\{\begin{array}{ll} [F_v^+ : \bbQ_l]n(n-1)/2 & \text{if }v \mid l \\ 0 & \text{if } v \nmid l.\end{array}\right.\end{equation*}
Let $q$ be the larger of $\dim_k H^1_{\cL^\perp, T}(G_{F^+, S}, \ad \overline{r} (1))$ and $q_0$. 

Then for any $N \in \bbZ_{\geq 1}$ we can find $(Q, \widetilde{Q}, \{ \overline{\psi}_v\}_{v \in Q})$ as in the above definition, such that \begin{itemize} \item $\# Q = q \geq q_0$.
\item If $v \in Q$ then $\bbN v \equiv 1 \mod l^N$.
\item $R^{\square_T}_{\cS_Q}$ can be topologically generated over $R^\text{loc}_{\cS, T} = R^\text{loc}_{\cS_Q, T}$ by \begin{equation*} \# Q - \sum_{v \in T, v \mid l} [F_v^+ : \bbQ_l]n(n-1)/2 - n \sum_{v \mid \infty} (1 + \chi(c_v))/2\end{equation*} elements.
\end{itemize}
\end{proposition}
\begin{proof} Given $(Q, \widetilde{Q}, \{ \overline{\psi}_v\}_{v \in Q})$ there is an exact sequence
\begin{equation*}\xymatrix{0 \ar[r] & H^1(G_{F^+, S}, \ad\overline{r}(1)) \ar[r] & H^1(G_{F^+, S \cup Q}, \ad \overline{r}(1)) \ar[r] &}\end{equation*}
\begin{equation*}\xymatrix { \ar[r] & \bigoplus_{v \in Q}  H^1(I_{F_\wv}, \ad \overline{s} (1)) \oplus H^1(I_{F_\wv}, \ad \overline{\psi} (1)),}\end{equation*}
the last arrow given by restriction. Now, for $v \in Q$,  $L_v$ is the subspace of 
\begin{equation*}H^1(G_{F_\wv}, \ad \overline{r}(1)) = H^1(G_{F_\wv}, \ad \overline{s}(1)) \oplus H^1(G_{F_\wv}, \ad \overline{\psi}(1))\end{equation*}
whose projection to $H^1(I_{F_\wv}, \ad \overline{s}(1))$ is trivial and whose projection to $H^1(I_{F_\wv}, \ad \overline{\psi}(1))$ actually takes values in 
\begin{equation*}H^1(I_{F_\wv}, Z(\overline{\psi})(1)),\end{equation*} 
where $Z(\overline{\psi}) \subset \ad \overline{\psi}$ is the subspace of scalar matrices.

Recall that the complement $L_v^\perp$ of $L_v$ is taken with respect to the pairing $\ad \overline{r} \times \ad \overline{r}(1) \rightarrow k(1)$ given by $(A, B) \mapsto \tr AB$. It follows that for $v \in Q$, $L_v^\perp$ is identified with the subspace of unramified cohomology classes in $H^1(G_{F_\wv}, \ad \overline{r} (1))$ whose projection to $H^1(G_{F_\wv}, \ad \overline{\psi} (1))$ actually takes values in $H^1(G_{F_\wv}, \ad^0 \overline{\psi} (1))$. In particular, the group $H^1_{\cL(Q)^\perp, T}(G_{F^+, S\cup Q},\ad \overline{r} (1)) $ is the kernel of the map
\[ H^1_{\cL^\perp, T}(G_{F^+,S},\ad \overline{r}(1))\rightarrow \oplus_{v \in Q} k\]
given by $[\phi] \mapsto (\tr e_{\Frob_\wv, \alpha_v} \phi(\Frob_\wv))_{v \in Q}$. By the previous lemma, we'll be done if for any non-zero cohomology class \begin{equation*}[\phi] \in H^1_{\cL^\perp, T}(G_{F^+,S},\ad \overline{r}(1)),\end{equation*} we can find a place $v$ of $F^+$ such that $v$ splits completely in $F(\zeta_{l^N})$ and that 
\begin{equation*}\tr e_{\Frob_\wv, \alpha_v} \phi(\Frob_\wv) \neq 0\end{equation*}
 for one of the places $\wv$ of $F$ above $v$ and choice of eigenvalue $\alpha_v$, and such that $\overline{r}(\Frob_\wv)$ acts semisimply on its $\alpha_v$-eigenspace.
 
 By the Chebotarev density theorem, it suffices to find an element $\sigma_0$ of $G_{F(\zeta_{l^N})}$ such that
 \begin{equation*}\tr e_{\sigma_0, \alpha} \phi(\sigma_0) \neq 0\end{equation*}
 for some eigenvalue $\alpha$ of $\overline{r}(\sigma_0)$.
 
 Let $L / F(\zeta_{l^N})$ be the extension cut out by $\ad \overline{r}$. We have an exact sequence
 \begin{equation*}\xymatrix{0 \ar[r] & H^1(G_{F^+(\zeta_{l^N})}, \ad \overline{r}(1)) \ar[r] & H^1(G_L, \ad\overline{r}(1))^{G_{F^+(\zeta_{l^N})}}.}\end{equation*}
 
 Thus $\phi(G_L)$ is a non-zero $G_{F^+(\zeta_{l^N})}$-submodule of $\ad \overline{r}(1)$. By hypothesis, there exists $\sigma \in G_{F(\zeta_{l^N})}, \alpha$ such that $\overline{r}(\sigma)$ acts semisimply on its $\alpha$-eigenspace and such that
 \begin{equation*}\tr e_{\sigma, \alpha} \phi(G_L) \neq 0.\end{equation*}
 If
  \begin{equation*}\tr e_{\sigma, \alpha} \phi(\sigma) \neq 0\end{equation*}
then we're done. So we can assume that in fact
 \begin{equation*}\tr e_{\sigma, \alpha} \phi(\sigma) =  0.\end{equation*}
For $\tau \in G_L$, $\overline{r}(\tau \sigma)$ is a scalar multiple of $\overline{r}(\sigma)$ and we have $\phi(\tau \sigma) = \phi(\tau) + \phi(\sigma)$, hence
 \begin{equation*}\tr e_{\sigma, \alpha} \phi(\tau \sigma) =  \tr e_{\sigma, \alpha} \phi(\tau).\end{equation*}
 Therefore we can choose $\sigma_0 = \tau \sigma$ for any $\tau \in G_L$ such that
 \begin{equation*}\tr e_{\sigma, \alpha} \phi(\tau) \neq 0.\end{equation*}
This concludes the proof.
\end{proof}

\section{Smooth representation theory of $q$-adic $\GL_n$}
We introduce some more notation. Let $F$ be a finite extension of $\bbQ_p$, and choose a uniformizer $\varpi$. Let $q$ denote the size of the residue field $\frf$ of $F$. We write $G = \GL_n(F)$, $B$ for its standard Borel subgroup. We fix a partition $n = n_1 + n_2$, and let $P$ be the standard parabolic (containing B) corresponding to this partition. Let $P = MN$ be its Levi decomposition. We will use $P$ interchangeably to denote the algebraic subgroup of $\GL_n$ or its $F$-points. This should not cause any confusion. We let $\overline{P}$ denote the opposite parabolic of $P$.

We write $K = \GL_n(\cO_F)$, and $\frp$ for the parahoric subgroup of $G$ corresponding to those elements of $K$ whose reduction modulo $\varpi$ lies in $P(k)$. Let $\ffrm = M(\cO_F)$. We write $\Iw$ for the standard Iwahori subgroup of $G$. We consider smooth representations of $G$ over an algebraically closed field $C$ of characteristic zero.

We use $\Ind$ to denote standard unnormalized parabolic induction and $\nInd$ to denote the normalized parabolic induction $\nInd_P^G \sigma = \Ind_P^G \sigma \delta_P^\frac{1}{2}$, where $\delta_P$ is the modulus character of $P$. We use $\pi \mapsto \pi_N$ to denote the normalized Jacquet functor which is right adjoint to $\nInd$. In particular, if $\pi = \nInd_B^G \chi_1 \otimes \dots \otimes \chi_n$ is irreducible then
\begin{equation*}r_l(\pi)^\vee(1-n) = \oplus_i(\chi_i | \cdot |^{(1-n)/2)}) \circ \Art_F^{-1}.\end{equation*}

If $\psi$ is an unramified character, we write $\St_m(\psi)$ for the unique generic subquotient of
\begin{equation*}\nInd_{B_m}^{\GL_m} \psi \otimes \psi |\cdot| \otimes \dots \otimes \psi |\cdot|^{m-1},\end{equation*}
where $B_m \subset \GL_m$ is the standard Borel subgroup. We refer to this as a \emph{Steinberg} representation.

Write $W$ for the Weyl group of $G$. Given a parabolic subgroup $Q$ of $G$, we write $W_Q \subset W$ for the Weyl group of its Levi factor. Recall that $W_Q \backslash W \slash W_P$ has a canonical set of representatives, which we will denote $[W_Q \backslash W \slash W_P]$, given by taking in each double coset the element of minimal length (cf. the first section of \cite{Cas95}).

\begin{lemma} Let $Q$ be the standard parabolic corresponding to the partition $n = m_1 + \dots + m_r$. Then there is a bijection 
\begin{equation*}W_Q \backslash W \slash W_P \cong \text{the set of partitions }m_i = n_1^i + n_2^i, i=1, \dots, r\end{equation*} \begin{equation*} \text{ such that } \sum_i n_1^i = n_1 \text{ and } \sum_i n_2^i = n_2,\end{equation*}
given as follows. Let 
\begin{equation*}M_i = \{ m_1 + \dots + m_{i-1} + 1, \dots, m_1 + \dots + m_{i-1} + m_i\},
\end{equation*}
$ N_1 = \{1, \dots, n_1\},$ and $ N_2 = \{n_1 + 1, \dots, n\}$. Then we take
\begin{equation*}W_Q w W_P \mapsto \text{the partition }m_i = (\# M_i \cap wN_1) + (\#M_i \cap wN_2).\end{equation*}
(Here $W$ acts in the natural manner on $\{1, \dots, n\}$.)
\end{lemma}

Given $Q$ as in the lemma, we write $L_i$ for the $i^\text{th}$ block factor of its Levi subgroup. Given a partition $m_i = n_1^i + n_2^i$, corresponding to a double coset $W_Q w W_P$, we write $\frp_i^w$ for the parahoric subgroup of $L_i$ corresponding to those elements of $L_i(\cO_F)$ whose reduction modulo $\varpi$ lies in the two-block standard parabolic given by $m_i = n_1^i + n_2^i.$

\begin{lemma}\label{invariants} Let $Q$ be as above. \begin{enumerate} \item If $w \in [W_Q \backslash W \slash W_P]$, then $L_i \cap w \frp w^{-1} = \frp^w_i.$
\item For each $i$, let $\pi_i$ be a smooth representation of $L_i$. Then \begin{equation*}(\nInd_Q^G \pi_1 \otimes \dots \otimes \pi_r)^\frp = \oplus_{w \in [W_Q \backslash W \slash W_P]} \pi_1^{\frp^w_1} \otimes \dots \otimes \pi_r^{\frp^w_r}.\end{equation*}
\end{enumerate}
\end{lemma}
\begin{proof} We recall that an element $w \in W$ is of minimal length in its double coset if and only if it is of minimal length in $W_Q w$ and $w W_P$. If $R$ is the parabolic corresponding to a partition $n = a_1 + \dots + a_s$ then an element is minimal in its left $W_R$-coset if and only if it is order preserving on the sets 
\begin{equation*}
\{a_1 + \dots + a_{i-1} + 1, \dots, a_1 + \dots + a_{i-1} + a_{i}\}
\end{equation*}
 for each $i=1, \dots, s$. This is well known, but can be deduced easily from e.g. the results on p. 20 of \cite{Bjo05}. The first part now follows from a direct calculation (alternatively, one could apply \cite{Cas95}, Proposition 1.3.3).

For the second part, we use the Bruhat decomposition to deduce that
\begin{equation*}G = \coprod_{w \in [W_Q \backslash W \slash W_P]} Qw\frp.\end{equation*}
Then we have from the definition of induction
\begin{equation*}(\nInd_Q^G \pi_1 \otimes \dots \otimes \pi_r)^\frp = \oplus_{w \in [W_Q \backslash W \slash W_P]}(\pi_1 \otimes \dots \otimes \pi_r)^{Q \cap w\frp w^{-1}},\end{equation*}
and applying the first part of the lemma gives the result.
\end{proof}
Let $\pi$ be an irreducible smooth representation of $G$. If $\pi$ is generic, then it has an expression $\pi = \nInd^G_Q \pi_1 \otimes \dots \otimes \pi_r$ as above, where each $\pi_i$ is an irreducible essentially square-integrable representation of $L_i$ (cf. \cite{Pra08}).

\begin{corollary} Suppose that $\pi^\frp \neq 0$. Then there exist unramified characters $\chi_1, \dots, \chi_s$ and $\psi_1, \dots, \psi_t$ of $F^\times$ such that \begin{equation*}\pi = \nInd^G_Q \chi_1 \otimes \dots \otimes \chi_s \otimes \St_2(\psi_1) \otimes \dots \otimes \St_2(\psi_t).\end{equation*}
\end{corollary}
\begin{proof} Since $\Iw \subset \frp$, $\pi$ must be a quotient of an unramified principal series. Therefore if we write 
\begin{equation*}\pi = \nInd^G_Q \pi_1 \otimes \dots \otimes \pi_r\end{equation*}
then each $\pi_i$ is either an unramified character or a Steinberg representation $\St_n(\psi)$. The representation $\St_n(\psi)$ has invariants under a two-block parahoric subgroup only if $n=2$, so by the previous lemma the factors $\pi_i$ must have the given form. Since $\pi$ is generic the normalized induction is independent of the order of the factors, hence we can write $\pi$ as in the statement of the corollary.
\end{proof}
Following \cite{Vig98}, we introduce Hecke algebras $H_\frp = \cH_\bbZ(G, \frp)$ and $H_\ffrm = \cH_\bbZ(M, \ffrm)$.  We recall from section II of that paper that $H_\ffrm$ has a subalgebra $H_\ffrm^-$ consisting of elements with `negative' support. In particular, this subalgebra contains the Hecke operators $T^j$ defined below. Moreover, there is a natural homomorphism $t : H_\ffrm^- \otimes C \rightarrow H_\frp \otimes C$, given on characteristic functions of double cosets by $\chi_{\ffrm g \ffrm} \mapsto \delta_{\overline{P}}(g)^\frac{1}{2} \chi_{\frp g \frp}$.

\begin{proposition} \label{intertwining} Let $\pi$ be an admissible representation of $G$, and let $q : \pi^\frp \rightarrow \pi_N^\ffrm$ be the natural projection. Then $q$ is an isomorphism, and for  all $v \in \pi^\frp, h \in H_\ffrm^-$, we have \begin{equation*}q(t(h) \cdot v) = h \cdot q(v).\end{equation*}
\end{proposition}
\begin{proof}
Since $\frp$ has an Iwahori decomposition with respect to $P$, $q$ is surjective (\cite{Vig98}, II.10.1). It remains to prove injectivity. There is a commutative diagram
\begin{equation*}\xymatrix{ \pi^\frp \ar[r]\ar[d] & \pi^{\Iw}\ar[d] \\
(\pi_N)^J \ar[r] & (\pi_R)^{T_0},}\end{equation*}
where $R$ is the radical of $B$ and $T_0$ is the maximal compact subgroup of $T$, and $J$ is the standard Iwahori subgroup of $M = \GL_{n_1} \times \GL_{n_2}$. The right vertical arrow is an isomorphism  (\cite{Vig98}, II.7). Similarly the bottom horizontal arrow is an isomorphism. It follows that the left vertical arrow is injective. 

The second part of the proposition is \cite{Vig98}, Lemma II.9 (but note that we use normalized restriction here).
\end{proof}

We will be interested in the following Hecke operators. For $j=1, \dots, n_2$, let $\alpha_j$ be the $n_2 \times n_2$ matrix 
\begin{equation*}\diag(\underbrace{\varpi, \dots, \varpi}_j, 1, \dots, 1).\end{equation*}
 We let $V^j \in \cH_\frp$ be defined as the double coset operator 
 \begin{equation*}\left[\frp \left(\begin{array}{cc} 1_{n_1} & 0 \\ 0 & \alpha_j \end{array} \right)\frp\right].\end{equation*} 
 Similarly $T^j$ is the double coset operator 
 \begin{equation*}\left[\ffrm \left(\begin{array}{cc} 1_{n_1} & 0 \\ 0 & \alpha_j \end{array} \right) \ffrm\right].\end{equation*}
One computes easily that $V^j =  q^{jn_1/2} t(T^j) .$ Note that since $t$ is an algebra homomorphism, the fact that the $T^j$ commute implies that the operators $V^j$ must also commute.

\begin{proposition} \begin{enumerate} \item Let $w \in [W_Q \backslash W \slash W_P]$, corresponding to a partition $m_i = n_1^i + n_2^i$. This can be viewed as a partition $n_1 = \sum_i n_1^i, n_2 = \sum_i n_2^i$. Then $w^{-1} Q w \cap M$ is the standard parabolic of $M = \GL_{n_1} \times \GL_{n_2}$ corresponding to this partition.
\item Let $\pi = \nInd^G_Q \pi_1 \otimes \dots \otimes \pi_r$, where each $\pi_i$ is an  admissible representation of $L_i$. Then \begin{equation*}\pi_N^\text{ss} \cong \oplus_{w \in [W_Q \backslash W \slash W_P]} \nInd^M_{w^{-1}Qw \cap M} w^{-1} (\pi_1 \otimes \dots \otimes \pi_r)_{L \cap w N w^{-1}}.\end{equation*} 
\end{enumerate}
\end{proposition}
\begin{proof} The first part is \cite{Cas95}, Proposition 1.3.3. The second is \cite{Ber77}, Lemma 2.12.
\end{proof}
We deduce the following.
\begin{corollary} Suppose that $\pi$ is generic and $\pi^\frp \neq 0$. Then we have a presentation\begin{equation*}\pi = \nInd^G_Q \chi_1 \otimes \dots \otimes \chi_s \otimes \St_2(\psi_1) \otimes \dots \otimes \St_2(\psi_t)\end{equation*} as above. There is an isomorphism 
\begin{equation*}\pi_N^{ss} \cong \sigma \oplus \bigoplus_\cS \nInd_{B\cap M}^M \left(\bigotimes_{i \not\in \cS} \chi_i \otimes \bigotimes_{j=1}^t( \psi_j \otimes |\cdot|) \otimes \bigotimes_{i \in \cS} \chi_i \otimes \bigotimes_{j=1}^t \psi_j \right),\end{equation*}
where $\sigma$ is a representation with no $\ffrm$-invariants. Here the sum is over subsets $\cS \subset \{1, \dots, n -2t \}$ of order $n_2 - t$. In particular, when $t=0$, we have
\begin{equation*}\pi_N^{ss} \cong \bigoplus_{\cS}  \nInd^M_{B \cap M} \left(\bigotimes_{i \not\in \cS} \chi_i\right) \otimes \left(\bigotimes_{i \in \cS} \chi_i \right).\end{equation*}
\end{corollary}
\begin{proof} The main point to note is that
\begin{equation*}\St_2(\psi)_N  = \psi |\cdot| \otimes \psi.\end{equation*}
\end{proof}

Since passing to $\ffrm$-invariants is exact, the above computes $(\pi_N^{ss})^\ffrm = (\pi_N^\ffrm)^{ss}$ as a $H_\ffrm$-module. Combining this with Proposition \ref{intertwining}, along with the analogous computation for unramified Hecke operators on $\GL_n$, we can compute the eigenvalues of the operators $V^j$ on $\pi^\frp$ with their multiplicities.
 
 \begin{corollary} Remain in the above situation. Let $1 \leq j \leq n_2$. Let $\cA$ be the set of subsets of $\{1, \dots, n-2t\}$ of order $n_2 - t$. View $\pi^\frp$ as a module for the commutative algebra $C[V^1, \dots, V^{n_2}]$. Then $(\pi^\frp)^{ss}$ is a direct sum of 1-dimensional modules indexed by $\cA$. If $\cS \in \cA$ then $V^j$ acts on the corresponding line as
 \begin{equation*}q^{j( n -j)/2}\sum_{\substack{J_1 \subset \cS \\ J_2 \subset \{1, \dots, t\} }} \prod_{a \in J_1} \chi_a(\varpi) \prod_{b \in J_2} \psi_b(\varpi),\end{equation*} 
with the sum ranging over subsets $J_1, J_2$ with $\#J_1 + \# J_2 = j$. 
 In particular, when $t=0$, $V^j$ acts on the line corresponding to $\cS$ as
 \begin{equation*}q^{j( n -j)/2}\sum_{\substack{J \subset \cS\\ \#J = j}} \prod_{a \in J} \chi_a(\varpi).\end{equation*}
 \end{corollary}
We now specialize to the case $C = \overline{\bbQ}_l$, and $q \equiv 1 \mod l$. Fix also a subfield $K \subset C$ finite over $\bbQ_l$, with ring of integers $\cO$, maximal ideal $\lambda$, and residue field $k$. We suppose that $l$ is odd. Then $\cO$ contains a square root of $q$, and we fix a choice of square root congruent to 1 modulo $\lambda$.

\begin{proposition} Let $1 \leq j \leq n_2$. There is a unique way to functorially associate to every monic polynomial $P(X)$ of degree $n$ with coefficients in an $\cO$-algebra $A$ another monic polynomial $P_j(X)$ such that when $A$ is an algebraically closed field of characteristic zero and
\begin{equation*}P(X)=\prod_{i=1}^n(X-\alpha_i),\end{equation*}
the roots of $P_j(X)$ with multiplicities are precisely the
\begin{equation*}q^{j(1-j)/2} \sum_{\substack{J \subset \cS\\ \#J = j}} \prod_{a \in J} \alpha_a\end{equation*}
as $\cS$ ranges over subsets of $\{1, \dots, n\}$ of order $n_2$. 
\end{proposition}
(The point of this is that when 
\begin{equation*}\pi = \chi_1 \boxplus \dots \boxplus \chi_n\end{equation*}
and  
\begin{equation*}\rho = r_l(\pi)^\vee(1-n)\end{equation*}
and $P(X)$ is the characteristic polynomial of $\rho(\Art_F(\varpi))$, the characteristic polynomial of $V^j$ on $\pi^\frp$ divides $P_j$ (and equality holds when $\pi$ is generic)).

\begin{proposition}\label{mainprop} Let $R$ be a complete local $\cO$-algebra with residue field $k$. Let $\Pi$ be a smooth $R[G]$-module, and suppose that for every open compact subgroup $U \subset G$, $\Pi^U$ is a finite free $\cO$-module. We write $R_U$ for the image of $R$ in $\End_\cO \Pi^U$. Suppose that $\Pi \otimes_\cO \barQl$ is a semisimple $\barQl[G]$-module, and that every irreducible constituent is generic. Suppose that each $R_U \otimes_\cO \barQl$ is a semisimple algebra.

Suppose there exists a continuous representation $\rho : G_F \rightarrow \GL_n(R_\frp)$ such that for any homomorphism $\varphi : R_\frp \rightarrow \barQl$, and any irreducible constituent $\pi$ of the representation generated by $\Pi^\frp \otimes_{R_\frp, \varphi} \barQl$, there is an isomorphism
\begin{equation*} (r_l(\pi)^\vee(1-n))^{ss} \cong (\rho \otimes_{R_\frp, \varphi} \barQl)^{ss}.\end{equation*}
Fix a Frobenius lift $F$, corresponding to uniformizer $\varpi$ under the local Artin map. Suppose that $\overline{\rho} = \rho \mod \ffrm_R$ is unramified, and that $\overline{\rho}(F)$  has an eigenvalue $\overline{\alpha}$ of multiplicity $n_2$. Let $P(X)$ be the characteristic polynomial of $\rho(F)$, and let $P_j(X)$ be as above. By Hensel's lemma, we can factor $P_j(X) = Q_j(X)R_j(X)$, where 
\begin{equation*}R_j(X) \equiv \left(X-\binom{n_2}{j}\overline{\alpha}^j\right)^{k_j} \mod \ffrm_R\end{equation*}
 and 
 \begin{equation*}Q_j\left(\binom{n_2}{j}\overline{\alpha}^j\right) \not\equiv 0 \mod \ffrm_R.\end{equation*}
 Set 
\begin{equation*}\pr_\varpi = \prod_{j=1}^{n_2} Q_j(V^j) \in \End_R \Pi^\frp.\end{equation*}
Then $\pr_\varpi$ induces an isomorphism
\begin{equation*}\Pi^K \overset{\sim}{\rightarrow} \pr_\varpi \Pi^\frp \subset \Pi^\frp\end{equation*}
of $R$-modules.
\end{proposition}
\begin{proof} First we show that if $\pi$ is an irreducible constituent of the representation generated by some  $\Pi^\frp \otimes_{R_\frp, \varphi} \barQl$ and $\pr_\varpi \pi^\frp \neq 0$ then $\pi$ is unramified. If $\pi$ is ramified, then as above we have for some $t>0$
\begin{equation*}\pi \cong \nInd^G_Q \chi_1 \otimes \dots \otimes \chi_s \otimes \St_2(\psi_1) \otimes \dots \otimes \St_2(\psi_t).\end{equation*}
The eigenvalues of $(\rho \otimes_{R_\frp, \varphi} \barQl)(F)$ are, with multiplicities, 
\begin{equation*}\chi_1(\varpi), \dots, \chi_s(\varpi), \psi_1(\varpi), |\varpi| \psi_1(\varpi), \dots, \psi_t(\varpi), |\varpi| \psi_t(\varpi).
\end{equation*}
Moreover, these last elements are contained in the image of $R_\frp$ in $\barQl$.
 
Suppose first that $\overline{\alpha} \equiv \psi_j(\varpi) \mod \ffrm_R$ for some $j$.  (We view the image of $R$ in $\barQl$ as a quotient of $R$, which therefore has maximal ideal induced by $\ffrm_R$). Then, since $\psi$ and $\psi|\cdot|$ are congruent modulo $\ffrm_R$, $\pr_\varpi$ projects to a space where $V_j$ acts as the root of $P_j(X)$ corresponding to a set $\cS$ of eigenvalues including both $\psi_j(\varpi)$ and $\psi_j(\varpi) |\varpi|$. However, our earlier computation shows that the only lines that occur correspond to a set of eigenvalues containing only $\psi_j$. So this is impossible.

Suppose instead that $\overline{\alpha} \not\equiv \psi_j(\varpi) \mod \ffrm_R$ for any $j$. Then $\pr_\varpi$ maps into a line corresponding to a set of eigenvalues not containing any $\psi_j(\varpi)$. But one knows that each $\psi_j(\varpi)$ occurs in every line of $\pi^\frp$, so this is also impossible.

Returning to the situation of the proposition, it follows that  
\begin{equation*}\rank_\cO \Pi^K \geq \rank_\cO \pr_\varpi \Pi^\frp.\end{equation*}
Therefore to show that the map 
\begin{equation*}\pr_\varpi : \Pi^K \rightarrow \pr_\varpi \Pi^\frp\end{equation*}
is in fact an isomorphism, it will suffice to show that it is injective after $-\otimes_\cO \overline{k}.$

Suppose that it is not, and let $x \in \Pi^K \otimes_\cO \overline{k}$ be a non-zero vector in the kernel such that $\ffrm_R \cdot x = 0$. Let $N$ be an irreducible quotient of the admissible $\overline{k}[G]$-module generated by $x$. The following lemma shows that $N^K$ is 1-dimensional, generated by $x$, and that
\begin{equation*}\pr_\varpi N^K \neq 0.\end{equation*}
This contradiction concludes the proof.
\end{proof}

\begin{lemma} Let $\pi$ be an unramified irreducible smooth representation of $G$ over $\overline{k} = \bbF$. Then:
\begin{enumerate}
\item There is an isomorphism $\pi \cong \nInd_Q^G \chi_1 \circ \det \otimes \dots \otimes \chi_r \circ \det$ for some distinct unramified characters $\chi_1, \dots, \chi_r$, $Q$ corresponding to a partition $n = m_1 + \dots + m_r$. 
\item $\pi^K$ is 1-dimensional.
\item Set $n_2 = m_r$ and $n_1 = \sum_{i=1}^{r-1} m_i$. Let $P(X) = \prod_{i=1}^r(X-\chi_i(\varpi))^{m_i}$, and let $P_j(X)$ be as above. Factor each $P_j(X)=Q_j(X)R_j(X)$, with 
\begin{equation*}R_j(X) = \left(X-\binom{n_2}{j}\chi_{m_r}(\varpi)^j\right)^{k_j}\end{equation*}
and $R_j(X), Q_j(X)$ coprime. Then
\begin{equation*}\pr_\varpi = \prod_{j=1}^{n_2} Q_j(V^j)\end{equation*}
induces an isomorphism
\begin{equation*}\pi^K \overset{\sim}{\rightarrow} \pr_\varpi \pi^\frp \subset \pi^\frp.\end{equation*}
\end{enumerate}
\end{lemma}
\begin{proof}
The first part is \cite{Vig98}, Assertion VI.2. The second part is an easy calculation.

For the third part, we argue as follows. First we give a description of the Iwahori Hecke algebra $H_{\Iw} = \cH_\bbF(G, \Iw)=\cH_\bbZ(G, \Iw) \otimes_\bbZ \bbF$, following \cite{Vig96}, I.3.14. For $j=1, \dots, n$, we let 
\begin{equation*}a_j = \diag(\underbrace{\varpi, \dots, \varpi}_j, 1, \dots, 1).\end{equation*}
Then we let $A_j = [\Iw a_j \Iw]$ and $X^j = A_j (A_{j-1})^{-1}$. (This makes sense since each $A_j$ is invertible in $H_{\Iw}$). On the other hand, if $s_j$ is the permutation matrix corresponding to the simple transposition $(j, j+1)$ then we set $S^j = [\Iw s_j \Iw]$. $H_{\Iw}$ is generated by the $S^j$ and $X^j$. Moreover, it is canonically isomorphic to the group algebra of the group $W \ltimes \bbZ^n$, where the transpositions $S^j$ generate the $W$ factor and the operators $X^1, \dots, X^n$ form a basis for $\bbF[\bbZ^n] \subset \bbF[W \ltimes \bbZ^n]$. (Thus $W \cong S_n$ acts on $\bbZ^n$ by permuting basis vectors).

We note that $\pi^{\Iw}$ has a basis given by the functions $\varphi_w$ for $w \in [W_Q\backslash W \slash W_B]$, defined as follows: $\varphi_w$ has support $Qw \Iw$, and $\varphi_w(1)=1$. Define characters $\psi_j$ by 
\begin{equation*}\psi_1 \otimes \dots \otimes \psi_n = \chi_1 \otimes \dots \otimes \underbrace{\chi_i \otimes \dots \otimes \chi_i}_{m_i}\otimes  \dots \otimes \chi_r.\end{equation*}
Then $X^i$ acts on $\varphi_w$ as the scalar $\psi_{w(i)}(\varpi)$. This can be deduced as follows.

First, note that $\pi \subset \nInd_B^G \psi$ in a natural manner (i.e. $\pi$ is the set of functions $G \rightarrow \bbF$ which transform under the left translation by $Q$ in a suitable manner, hence a fortiori by $\psi$ under $B$). Each $\varphi_w$ is a sum of functions $\phi_{w_1}$ with support $Bw_1\Iw$, for $w_1 \in W_Q w$. By computing the action of generators of the Iwahori Hecke algebra, one checks that $X^i$ acts on $\phi_{w_1}$ as the scalar $\psi_{w_1(i)}(\varpi) = \psi_{w(i)}(\varpi)$; cf. Corollary 3.2 of \cite{Lan02}. (We note that $q = 1 \in \bbF$ is used here in an essential manner). The corresponding result now follows for the vectors $\varphi_w$.

There is a commutative diagram
\begin{equation*}\xymatrix{ \pi^\frp \ar[r] \ar[d]^{q_1} & \pi^{\Iw} \ar[d]^{q_2} \\ (\pi_N)^\ffrm \ar[r]^{q_3} & (\pi_R)^{T_0},}\end{equation*}
where $R$ is the unipotent radical of $B$ and $T_0$ is the maximal compact subgroup of $T$. Note that $q_1$ and $q_2$ are isomorphisms, while $q_3$ is injective. Repeatedly applying Lemma II.9 of \cite{Vig98}, we find that for all $x \in \pi^\frp$ we have
\begin{equation*}q_2 V^j(x) = q_3 q_1 V^j(x) = q_3 T^j q_1 (x) = \sum_{\substack{J \subset \{n_1 + 1, \dots, n\}\\ \#J = j}} Y^J q_3 q_1 (x)\end{equation*} \begin{equation*}=\sum_{\substack{J \subset \{n_1 + 1, \dots, n\}\\ \#J = j}} Y^J q_2 (x) = \sum_{\substack{J \subset \{n_1 + 1, \dots, n\}\\ \#J = j}} q_2 X^J (x),\end{equation*}
where $X^J = \prod_{i \in J} X^i$ and $Y^J =  [T_0 y_J T_0],$ where $y_J$ is the diagonal matrix with $(y_J)_{ii} = \varpi$ if $i \in J$, and $(y_J)_{ii} = 1$ otherwise. The third equality above comes from writing coset representatives for $T^j$ (cf. Proposition 4.1 of \cite{Man01}) and using the fact that $q=1$ in $\bbF$.

Hence $V^j (x) = \sum_J X^J (x)$, which is to say that the restriction of $\sum_J X^J$ to the space of $\frp$-invariants is equal to $V^j$. It follows that $\{ \varphi_w \}$ is a basis of simultaneous eigenvectors for the operators $V^j$. If $w \in [W_Q \backslash W \slash W_P]$ corresponds to the partition $m_i = n_1^i + n_2^i$, then the eigenvalue of $V^j$ on $\varphi_w$ is
\begin{equation*}s_j(\chi_1(\varpi), \dots, \underbrace{\chi_i(\varpi), \dots, \chi_i(\varpi)}_{n_2^i \text{ times}}, \dots, \chi_r(\varpi)),\end{equation*}
where $s_j$ is the $j^\text{th}$ symmetric polynomial of degree $n_2$.

Set $\varphi = \sum_w \varphi_w$; this vector spans $\pi^K$. It is now easy to see that $\pr_\varpi \varphi = \varphi_w$, where $w$ is the element of $[W_Q \backslash W \slash W_P]$ corresponding to the partition 
\begin{equation*}m_1 = m_1 + 0, \dots, m_{r-1} = m_{r-1} + 0, m_r = 0 + m_r = 0 + n_2,\end{equation*}
and moreover that $\varphi_w$ spans $\pr_\varpi \pi^\frp$.

\end{proof}

At this point we introduce another congruence subgroup $\frp_1 \subset \frp$ as follows. It is the kernel of the homomorphism
\begin{equation*}\xymatrix{\frp \ar[r] & P(\frf) \ar[r] & \GL_{n_2}(\frf) \ar[r]_\det & \frf^\times \ar[r] & \frf^\times(l),}\end{equation*}
where $\frf^\times(l)$ denotes the maximal $l$-power order quotient of $\frf^\times$. Thus $\frp / \frp_1 \cong \frf^\times(l)$. We define an extra Hecke operator on ${\frp_1}$-invariants: for $\alpha \in \cO_F^\times$, $A_\alpha = \diag(\alpha, 1, \dots, 1)$,
\begin{equation*}V_\alpha = \left(\left[\frp_{1}\left(\begin{array}{cc} 1_{n_1} & 0 \\ 0 & A_\alpha \end{array} \right)\frp_{1}\right]\right).\end{equation*}
We also have the operator 
\begin{equation*}V^j = \left[\frp_1 \left(\begin{array}{cc} 1_{n_1} & 0 \\ 0 & \alpha_j \end{array} \right)\frp_1\right].\end{equation*} 
Note that this now depends on the choice of $\varpi$, although we do not include $\varpi$ in the notation. At this point we have two operators denoted $V^j$, acting on the spaces $\pi^\frp$ and $\pi^{\frp_1}$, for any smooth representation $\pi$. One checks that the inclusion $\pi^\frp \subset \pi^{\frp_1}$ takes one operator to the other, so in fact there is no ambiguity.

\begin{lemma}\label{ramified}
Let $\pi$ be a generic irreducible smooth $K$-representation of $G$. Suppose that $\pi^{\frp_1}  \neq 0$, but $\pi^{\frp}=0$. Then $\pi^{\frp_1}$ is 1-dimensional and for any representation
\begin{equation*}\rho : G_F \rightarrow \GL_n(\cO)\end{equation*}
with $\rho^{ss} \cong r_l(\pi)^\vee(1-n)$, we in fact have
\begin{equation*}\rho^{F-ss} \cong r_l(\pi)^\vee(1-n)\end{equation*}
and $\rho$ has abelian image. We can write
\begin{equation*}\pi = \nInd_B^G \chi_1 \otimes \dots \otimes \chi_n\end{equation*}
where $\chi_1, \dots, \chi_{n_1}$ are unramified and the remaining characters are tamely ramified with identical restriction to inertia.
\end{lemma}
\begin{proof}
Consider the subgroups $U_1 \subset U_0 \subset GL_2(\cO_F)$ defined as follows:
\begin{equation*} U_1 = \left \{ \left(\begin{array}{cc} \ast & \ast \\ 0 & 1 \end{array} \right) \text{ mod } \varpi \right \}, U_0 = \left \{ \left(\begin{array}{cc} \ast & \ast \\ 0 & \ast \end{array} \right) \text{ mod } \varpi \right \}.
\end{equation*}
Since $\pi^{\frp_1} \neq 0$, we have a presentation
\begin{equation*}\pi = \nInd^G_Q \chi_1 \otimes \dots \otimes \chi_s \otimes \St_2(\psi_1) \otimes \dots \otimes \St_2(\psi_t),\end{equation*}
with for each $j$ $\St_2(\psi_j)^{U_1}\neq 0$. If $t > 0$ then by conductor considerations we see that each $\psi_j$ is unramified and $\St_2(\psi_j)^{U_1}=\St_2(\psi_j)^{U_0}$, and hence $\pi^\frp \neq 0$, a contradiction (compare \cite{Clo08}, proof of Lemma 3.1.5). It follows that $t = 0$. Now note that $\frp_1$ contains the subgroup $\Iw' \subset \Iw$ defined as follows:
\begin{equation*}\Iw' = \{ (a_{ij}) \in \Iw \text{ such that } \prod_{i=n_1 + 1}^n a_{ii} \equiv 1 \mod \varpi \}.\end{equation*}
We can now apply \cite{Clo08}, Lemma 3.1.6  (or more precisely its proof) and the genericity of $\pi$ to deduce that it must have the form stated above.

Since $\pi$ is generic, any Weil-Deligne representation $(r, N)$ with $r = r_l(\pi)^\vee(1-n)^{ss}$ necessarily has $N=0$. Applying this to $WD(\rho^{F-ss})$ gives the result, on noting that a representation of $W_F$ with open kernel is abelian if and only if its Frobenius-semisimplification is. (Here we use $WD$ to denote the associated Weil-Deligne representation).
\end{proof}

\begin{proposition}\label{hecke}
Let $R$ be a complete local $\cO$-algebra with residue field $k$. Let $\Pi$ be a smooth $R[G]$-module, and suppose that for every open compact subgroup $U \subset G$, $\Pi^U$ is a finite free $\cO$-module. We write $R_U$ for the image of $R$ in $\End_\cO \Pi^U$. Suppose that $\Pi \otimes_\cO \barQl$ is a semisimple $\barQl[G]$-module, and that every irreducible constituent is generic. Suppose that each $R_U \otimes_\cO \barQl$ is a semisimple algebra.

Suppose there exists a continuous representation $\rho : G_F \rightarrow \GL_n(R_{\frp_1})$ such that for any homomorphism $\varphi : R_{\frp_1} \rightarrow \barQl$, and any irreducible constituent $\pi$ of the representation generated by $\Pi^{\frp_1} \otimes_{R_{\frp_1}, \varphi} \barQl$, there is an isomorphism
\begin{equation*} (r_l(\pi)^\vee(1-n))^{ss} \cong (\rho \otimes_{R_{\frp_1}, \varphi} \barQl)^{ss}.\end{equation*}
Fix a Frobenius lift $F$, corresponding to uniformizer $\varpi$ under the local Artin map. Suppose that $\overline{\rho} = \rho \mod \ffrm_R$ is unramified, and that $\overline{\rho}(F)$  has an eigenvalue $\overline{\alpha}$ of multiplicity $n_2$.  Define $\pr_\varpi$ as above. Let $R'$ denote the image of $R$ in $\End_\cO(\pr_\varpi \Pi^{\frp_1}).$

Then we can decompose $\rho\otimes_{R_{\frp_1}}R' = s \oplus \psi$ canonically as a sum of two subrepresentations such that $s$ is unramified, $\psi$ is tamely ramified and its restriction to inertia acts as a scalar character $\phi$. Finally we have for every $\alpha \in \cO_F^\times$, $V_\alpha = \phi(\Art_F(\alpha))$ in $R' \subset \End_{\cO}(\pr_\varpi \Pi^{\frp_1}).$
\end{proposition}
\begin{proof}  Let $\varphi, \pi$ be as in the statement of the proposition. Proposition \ref{mainprop} and Lemma \ref{ramified} show that if $\pr_\varpi \pi^{\frp_1} \neq 0$, then $\pi$ is either unramified or a ramified principal series, and in either case the image of $\rho \otimes_{R_{\frp_1},\varphi} \barQl$ is abelian. Let $\rho' = \rho \otimes_{R_{\frp_1}} R'.$

Let $P$ be the characteristic polynomial of $\rho'(F)$, and factor $P(X)=A(X)B(X)$, where ${B}(\overline{\alpha})= 0$ and $\overline{A}$, $\overline{B}$ are coprime. Then we decompose
\begin{equation*}R'^n = B(F)R'^n \oplus A(F)R'^n = s \oplus \psi, \text{ say.}\end{equation*}
Since $\rho'$ has abelian image, this decomposition is $\rho'$-invariant.

To prove the rest of the proposition, we can replace $R'$ with its image in $\barQl$, for some $\varphi, \pi$ with $\pr_\varpi \pi^{\frp_1} \neq 0$. We replace $\rho'$ with $\rho' \otimes_{R', \varphi} \barQl$. If $\pi$ is unramified then $s$ and $\psi$ are both unramified and $V_\alpha$ acts trivially on $\pr_\varpi \pi^{\frp_1}$.

If on the other hand $\pi$ is ramified then we must have
\begin{equation*}\pi \cong \nInd^G_B \chi_1 \otimes \dots \otimes \chi_n\end{equation*}
where $\chi_1, \dots, \chi_{n_1}$ are unramified and the remaining characters are ramified with equal restriction to inertia. Now $\pr_\varpi \pi^{\frp_1} \neq 0$ implies that $\pr_\varpi \pi^{\frp_1} = \pi^{\frp_1}$. One now computes that $V^j$ acts on the one dimensional space $\pi^{\frp_1}$ as the scalar
\begin{equation*}q^{j( n -j)/2}\sum_{\substack{J \subset \cS\\ \#J = j}} \prod_{a \in J} \chi_a(\varpi),\end{equation*}
$\cS = \{ n_1 +1 , \dots, n\}$. On the other hand $V^j$ acts on $\pr_\varpi \pi^{\frp_1}$ as the scalar
\begin{equation*}q^{j( n -j)/2}\sum_{\substack{J \subset \cS'\\ \#J = j}} \prod_{a \in J} \chi_a(\varpi),\end{equation*}
$\cS'$ equal to the set of $i$ such that $\chi_i(\varpi) \equiv \overline{\alpha}$ mod $\ffrm_R$.

Since these are the same, we find that
\begin{equation*}s^{ss} \cong (\chi_1 \oplus \dots \oplus \chi_{n_1})|\cdot|^{(1-n)/2} \circ \Art_F^{-1}\end{equation*}
and that
\begin{equation*}\psi^{ss} \cong (\chi_{n_1+1} \oplus \dots \oplus \chi_n)|\cdot|^{(1-n)/2} \circ \Art_F^{-1}.\end{equation*}
The final line of the proposition follows from the computation of the action of $V_\alpha$ on $\pi^{\frp_1}$.
\end{proof}

\section{Automorphic forms on definite unitary groups}\label{twk}

The constructions in the first part of this section are now standard, cf. \cite{Ger09}, \cite{Gue10}. The main point of repeating them here is to convince the reader that they go through with the assumption that $l > n$ weakened to the assumption that $l$ is merely an odd prime. We also take some steps to remove the hypothesis that the open compact subgroup $U$ is ``sufficiently small''.

We suppose that $L$ is an imaginary $CM$ field such that $L/L^+$. Let $c$ denote the non-trivial element of $\Gal(L/L^+)$. Let $B$ denote the matrix algebra $M_n(L)$, and let $(\cdot)^\ast$ be an involution on $B$ of the second kind. We let $G$ be the associated unitary group of transformations $g \in B$ such that $g g^\ast = 1$.

Suppose now that $L/L^+$ is unramified at all finite places and that 4 divides $n[L^+ : \bbQ]$. Under these hypotheses, we can choose $(\cdot)^\ast$ so that
\begin{enumerate} \item For every finite place $v$ of $L^+$, $G$ is quasi-split at $v$.
\item For every infinite place $v$ of $L^+$, $G(L^+_v)\cong U_n(\bbR)$.
\end{enumerate}
We can find a maximal order $\cO_B \subset B$ such that $\cO_B^\ast = \cO_B$ and such that $\cO_{B,w}$ is a maximal order in $B_w$ for every place $w \in L$ split over $L^+$. This defines an integral model for $G$ over $\cO_{L^+}$, which we continue to denote as $G$. We do not prove these facts here, but refer instead to \cite{Clo08}, Section 3.3, with the set $S(B)$ taken to be empty.

Let $v$ be a finite place of $L^+$ which splits as $v = w w^c$ in $L$. Then we can find an isomorphism
\begin{equation*}\iota_v : \cO_{B,v} \rightarrow M_n(\cO_{L,v}) = M_n(\cO_{L_w}) \times M_n(\cO_{L_{w^c}}),\end{equation*}
such that $\iota_v(g^\ast) = {}^t\iota_v(g)^c.$ Projection to the first factor then gives rise to an isomorphism
\begin{equation*}\iota_w : G(\cO_{L_v^+}) \rightarrow \GL_n(\cO_{L_w}).\end{equation*}

Let $l$ be an odd prime number, and suppose that every place of $L^+$ above $l$ splits in $L$. Let $S_l$ denote the set of places of $L^+$ dividing $l$. For each place in $v \in S_l$ we choose a place $\wv$ of $L$ above it and denote the set of these by $\widetilde{S}_l$. 

Let $K$ be a finite extension of $\bbQ_l$ in $\overline{\bbQ}_l$, with ring of integers $\cO$ and residue field $k$. We write $\lambda$ for the maximal ideal of $\cO$. We will suppose $K$ large enough to contain the image of every embedding of $L$ in $\overline{\bbQ}_l$. 

Let $\widetilde{I}_l$ denote the set of embeddings $L \hookrightarrow K$ inducing a place in $\widetilde{S}_l$. To each $\lambda \in (\bbZ^n_+)^{\widetilde{I}_l}$ we associate a finite free $\cO$-module $M_\lambda$ as follows:
\begin{equation*}M_\lambda = \otimes_{\tau \in \widetilde{I}_l} M_{\lambda_\tau},\end{equation*}
where $M_{\lambda_\tau}$ is as constructed in \cite{Ger09}, Section 2.2. (It is the algebraic representation of $\GL_n/\cO$ with highest weight $\diag(t_1, \dots, t_n) \mapsto \prod_{i=1}^n t_i^{\lambda_{\tau, i}})$. 

Then $M_\lambda$ can be viewed as a continuous representation of the group $G(\cO_{L^+, l})$, via the product of the maps
\begin{equation*}\tau \circ \iota_{\wv(\tau)} :  G(\cO_{L^+_{ v(\tau)}})\rightarrow \GL_n(\cO_{L_{\wv(\tau)}}) \rightarrow \GL_n(\cO).\end{equation*}
(Here $\wv(\tau)$ and $v(\tau)$ are the places $L$ and $L^+$, respectively, induced by the embedding $\tau : L \hookrightarrow K$). Similarly $W_\lambda = M_\lambda \otimes_\cO K$ can be viewed as a continuous representation of the group $G(L^+_l)$.

Let $R$ be a finite set of finite places of $L^+$, disjoint from $S_l$ and containing only places which split in $L$. Let $T \supset S_l \cup R$ be a finite set of places of $L^+$ which split in $L$. For each $v \in T$ we choose a place $\wv$ of $L$ above it, extending our previous choice for $v \in S_l$. We suppose that $U = \prod_v U_v$ is an open compact subgroup of $G(\bbA^\infty_{L^+})$ such that $U_v \subset \iota_\wv^{-1} \Iw(\wv)$ for $v \in R$.  (We recall that for a place $\wv$ of $L$, $\Iw(\wv)$ is the subgroup of $\GL_n(\cO_{L_\wv})$ consisting of matrices whose image in $\GL_n(k(\wv))$ is upper triangular. We will also write $\Iw_1(\wv) \subset \Iw(\wv)$ for the subgroup of matrices whose image in $\GL_n(k(\wv))$ is upper-triangular and unipotent).

For each $v \in R$, we choose a character 
\begin{equation*}\chi_v = \chi_{v, 1} \times \dots \times \chi_{v, n} : \Iw(\wv)/\Iw_1(\wv) \rightarrow \cO^\times,\end{equation*}
the decomposition being with respect to the natural isomorphism 
\begin{equation*}\Iw(\wv)/\Iw_1(\wv) \cong (k(\wv)^\times)^n.\end{equation*}
 We set 
\begin{equation*}M_{\lambda, \{\chi_v\}} = M_\lambda \otimes_\cO \left(\bigotimes_{v \in R} \cO(\chi_v)\right) \text{ and }W_{\lambda, \{\chi_v\}} = M_{\lambda, \{\chi_v\}} \otimes_\cO K.\end{equation*}
These are representations of the groups $G(\cO_{L^+, l}) \times \prod_{v \in R} \Iw(\wv)$ and $G(L^+_l) \times \prod_{v \in R} \Iw(\wv)$, respectively.
\begin{definition}

Let $\lambda, U, \{\chi_v\}$ be as above. If $A$ is an $\cO$-module, and $U_v \subset G(\cO_{F^+_v})$ for $v \mid l$, we write $S_{\lambda, \{\chi_v\}}(U, A)$ for the set of functions 
\begin{equation*}f : G(L^+)\backslash G(\bbA^\infty_{L^+}) \rightarrow M_{\lambda,\{\chi_v\}} \otimes_\cO A\end{equation*}
such that for every $u \in U$, we have $f(gu) = u_{S_l \cup R}^{-1} f(g)$, where $u_{S_l \cup R}$ denotes the projection to $\prod_{v \in S_l \cup R} U_v$.

We write $S_{\lambda, \{ \chi_v\}}(\barQl)$ for the set of functions
\begin{equation*}f : G(L^+)\backslash G(\bbA^\infty_{L^+}) \rightarrow W_{\lambda,\{\chi_v\}} \otimes_K \barQl\end{equation*}
such that there exists an open compact subgroup $V$ such that for every $v \in V$, we have $f(gv) = v_{S_l \cup R}^{-1} f(g)$. This space receives an action of the group $G(\bbA_{L^+}^{\infty, R}) \times \prod_{v \in R} \Iw(\wv)$ via
\begin{equation*}(u \cdot f)(g) = (u_{S_l \cup R}) f(gu).\end{equation*}
Thus for $U$ as above, we have $S_{\lambda, \{ \chi_v\}}(\barQl)^U = S_{\lambda, \{ \chi_v\}}(U, \barQl)$.

If $R$ is empty then we write $S_{\lambda, \{\chi_v\}}(U, A) = S_\lambda(U, A)$.
\end{definition}
The spaces $S_{\lambda,\{\chi_v\}}(U, A)$ receive an action of the Hecke operators
\begin{equation*}T_w^j = \iota_w^{-1} \left( \left[ \GL_n(\cO_{L_w}) \left(\begin{array}{cc} \varpi_w 1_{j} & 0 \\ 0 & 1_{n-j} \end{array} \right)\GL_n(\cO_{L_w})\right]\right).\end{equation*}
Here $w$ is a place of $L$ split over $L^+$, not in $T$ and $\varpi_w$ is a uniformizer of $L_w$. We let $\bbT^T_{\lambda,\{\chi_v\}}(U, A)$ be the (commutative) $\cO$-subalgebra of $\End_\cO(S_{\{\chi_v\},\lambda}(U, A))$ generated by the operators $T_w^j$, $j=1, \dots n$, and $(T_w^n)^{-1}$ for $w$ as above. Again, if $R$ is empty, then we write 
\begin{equation*}\bbT^T_{\lambda,\{\chi_v\}}(U, A)=\bbT^T_{\lambda}(U, A).\end{equation*}

We briefly recall the relation between the spaces defined above and the space $\cA$ of automorphic forms on $G(L^+)\backslash G(\bbA_{L^+})$. 
\begin{proposition} Choose an isomorphism $\iota : \barQl \overset{\sim}{\rightarrow} \bbC$. For each $\tau \in \widetilde{I}_l$, we have $\iota\tau : F \hookrightarrow \bbC$. Let $\xi_{\lambda_{\iota \tau}}$ be the $\bbC$-representation of $G(L^+_{\tau |_{L^+}}) \cong U_n(\bbR)$ of highest weight $\lambda_{\tau}$, and let $\xi_{\iota \lambda}$ be the representation $\otimes_{\tau \in \widetilde{I}_l} \xi_{\lambda_{\iota \tau}}$ of $G(L^+_\infty)$. Then there is an $\iota$-linear isomorphism of $G(A_{L^+}^{\infty, R}) \times \prod_{v \in R} \Iw(\wv)$-modules
\begin{equation*}\iota : S_{\lambda, \{\chi_v\}}(\barQl) \overset{\sim}{\rightarrow} \Hom_{G(L^+_\infty)}\left((\otimes_{v \in R} \bbC(\iota \chi_v^{-1})) \otimes \xi_{\iota\lambda}^\vee, \cA\right).\end{equation*}
In particular, if $\Pi$ is a $G(\bbA_{L^+})$-constituent of $\cA$ with 
\begin{equation*}\Pi^\infty = \Hom_{G(L^+_\infty)}\left((\otimes_{v \in R} \bbC(\iota \chi_v^{-1})) \otimes \xi_{\iota\lambda}^\vee, \Pi \right) \neq 0,\end{equation*}
 then we can view $\iota^{-1} \Pi^{\infty}$ as a $G(\bbA_{L^+}^{\infty, R}) \times \prod_{v \in R} \Iw(\wv)$-constituent of $S_{\lambda, \{\chi_v\}}(\barQl).$
\end{proposition}
\begin{proof} The proof is almost identical to the first part of the proof of \cite{Clo08}, Proposition 3.3.2. 
\end{proof}

We end this section with two lemmas about these spaces, generalizing \cite{Clo08}, Lemma 3.3.1.

\begin{lemma}\label{suffsmall} Suppose that for all $t \in G(\bbA_{L^+}^\infty)$, the group 
\begin{equation*}t^{-1} G(L^+) t \cap U\end{equation*}
 contains no element of order $l$. Then the functor $A \mapsto S_{\lambda, \{\chi_v\}}(U, A)$ on $\cO$-modules is exact. 
\end{lemma}
\begin{proof} Write $G(\bbA_{L^+}^\infty) = \coprod_i G(L^+)t_jU$, a finite union. Then we have
\begin{equation*}S_{\lambda, \{\chi_v\}}(U, A) = \oplus_j (M_{\lambda, \{\chi_v\}} \otimes_\cO A)^{ t_j^{-1} G(L^+) t_j \cap U},\end{equation*}
\begin{equation*} f \mapsto (f(t_j))_j.\end{equation*}
Now note that for each $j$, $t_j^{-1} G(L^+) t_j \cap U$ is a finite group of order prime to $l$ (compare \cite{Gro99}, Proposition 4.3).
\end{proof}

\begin{lemma}\label{freeness}
 Suppose that for all $t \in G(\bbA_{L^+}^\infty)$, the group 
 \begin{equation*}t^{-1} G(L^+) t \cap U\end{equation*}
  contains no element of order $l$, and let $V \subset U$ be a normal open compact subgroup with $U/V$ abelian, of $l$-power order. The group $U/V$ operates on the space $S_{\lambda, \{ \chi_v \}}(U, A)$ via the diamond operators $u \mapsto [VuV]$. Then:
\begin{enumerate} \item The map $\tr_{U/V} : S_{\lambda, \{\chi_v\}}(V, A)_{U/V} \rightarrow S_{\lambda, \{\chi_v\}}(U, A)$ is an isomorphism.
\item $S_{\lambda, \{\chi_v\}}(V, \cO)$ is a free $\cO[U/V]$-module.
\end{enumerate}
\end{lemma}
\begin{proof} For the first part of the lemma, we can suppose that $U/V$ is cyclic; let $\sigma$ be a generator. By the previous lemma we can also assume that $A = \cO$. 

Let $\lambda^\vee$ be the weight defined by $\lambda^\vee_{\tau,i} = -\lambda_{\tau,  n  + 1 - i}$. Then we have $(W_{\lambda, \{\chi_v\}})^\vee \cong W_{\lambda^\vee, \{\chi_v^{-1}\}}$ as representations of $G(L^+_l) \times \prod_{v \in R} \Iw(\wv)$ (cf. \cite{Jan87}, Corollary II.2.5). Thus $(M_{\lambda, \{\chi_v\}})^\vee$ defines a lattice in $W_{\lambda^\vee, \{\chi_v^{-1}\}}$ which is invariant under the action of $G(\cO_{L^+,l}) \times \prod_{v \in R} \Iw(\wv)$. If $W$ is a compact open subgroup of $G(\bbA^\infty_{L^+})$, and $A$ is an $\cO$-module, let us temporarily write $S_{\lambda^\vee, \{ \chi_v^{-1} \}}(W, A)$ for the set of functions
\begin{equation*}f : G(L^+)\backslash G(\bbA^\infty_{L^+}) \rightarrow (M_{\lambda, \{\chi_v\}})^\vee \otimes_\cO A\end{equation*}
such that for every $w \in W$, we have $f(gw) = w_{S_l \cup R}^{-1} f(g)$.

We now have a pairing $S_{\lambda, \{\chi_v\}}(V, \cO) \times S_{\lambda^\vee, \{\chi^{-1}_v\}}(V, \cO) \rightarrow \cO$ given by the formula
\begin{equation*}(f,g)_V = \sum_{t \in G(L^+)\backslash G(\bbA^\infty_{L^+}) \slash V}  \frac{\langle f(t), g(t) \rangle} {\# (t^{-1} G(L^+) t \cap V)},\end{equation*}
where $\langle, \rangle$ is the natural duality pairing. We use the same formula to define a pairing $(,)_U$ on the spaces of level $U$. In fact, these are perfect pairings, and for $u \in U$ we have
\begin{equation*}([VuV]f, g)_V = (f, [Vu^{-1}V]g)_V.\end{equation*}
 Moreover, the diagram
\begin{equation*}\begin{array}{ccccc}S_{\lambda, \{\chi_v\}}(V, \cO)  & \times & S_{\lambda^\vee, \{\chi^{-1}_v\}}(V, \cO) & \rightarrow &  \cO \\
\downarrow {\tr_{U/V}} & & \rotatebox[origin=c]{90}{$\subset$} & & \parallel \\
S_{\lambda, \{\chi_v\}}(U, \cO)  & \times & S_{\lambda^\vee, \{\chi^{-1}_v\}}(U, \cO) & \rightarrow & \cO
\end{array}\end{equation*}
commutes. The map $\tr_{U/V}$ being an isomorphism is equivalent to the exactness of the sequence
\begin{equation*}\xymatrix@1{0 \ar[r] & (\sigma - 1) S_{\lambda, \{\chi_v\}}(V, \cO) \ar[r] & S_{\lambda, \{\chi_v\}}(V, \cO)  \ar[r]^{\tr_{U/V}} & \ar[r] S_{\lambda, \{\chi_v\}}(U, \cO) \ar[r] & 0.}\end{equation*}
Applying Pontryagin duality, and the pairing constructed above, this is equivalent to the natural map
\begin{equation*}S_{\lambda^\vee, \{\chi^{-1}_v\}}(U, \cO) \otimes K/\cO \rightarrow (S_{\lambda^\vee, \{\chi^{-1}_v\}}(V, \cO) \otimes K/\cO)^{\sigma = 1}\end{equation*}
being an isomorphism. But by the same argument as in the proof of previous lemma, this is just the map
\begin{equation*}S_{\lambda^\vee, \{\chi^{-1}_v\}}(U, K/\cO) \rightarrow (S_{\lambda^\vee, \{\chi^{-1}_v\}}(V, K/\cO))^{U/V},\end{equation*}
and the right hand side equals the left hand side by definition.

The second part of the lemma now follows from the first part as follows. Let $r = \dim_k S_{\lambda, \{\chi_v\}}(U, k)$. By Nakayama's lemma, there is a surjection
\begin{equation*}\cO[U/V]^r \rightarrow S_{\lambda, \{\chi_v\}}(V, \cO)\end{equation*}
of $\cO[U/V]$-modules. To show that this is an isomorphism, it's enough to check that 
\begin{equation*}\dim_K S_{\lambda, \{\chi_v\}}(V, K) = \#(U/V) r = \#(U/V) \dim_K S_{\lambda, \{\chi_v\}}(U, K).\end{equation*}
But this follows from the fact that for any $t$ we have $tG(L^+)t^{-1} \cap U = t G(L^+) t^{-1} \cap V$, along with the description of  $S_{\lambda, \{\chi_v\}}(V, K)$ given in the proof of the previous lemma.
\end{proof}

\subsection*{Galois representations} Keep the assumptions of the previous section.
\begin{theorem}\label{galois} Suppose that $\pi$ be an irreducible $G(\bbA^{\infty,R}_{L^+}) \times \prod_{v \in R} \Iw(\wv)$-constituent of $S_{\lambda, \{ \chi_v\}}(\barQl)$. Then there exists a continuous semisimple representation
\begin{equation*}r_l(\pi) : G_L \rightarrow \GL_n(\barQl)\end{equation*}
such that:
\begin{enumerate} \item If $v \not\in S_l$ is a place of $L^+$ which splits as $v = w w^c$ in $L$, then
\begin{equation*}(r_l(\pi)|_{G_{L_w}})^{ss} \cong (r_l(\pi_v \circ \iota_w^{-1}))^{ss}.\end{equation*}
\item $r_l(\pi)^c \cong r_l(\pi)^\vee(1-n).$
\item If $v$ is an inert place and $\pi_v$ has a fixed vector for a hyperspecial maximal compact in $G(L^+_v)$ then $r_l(\pi)$ is unramified at $v$.
\item If $v \in R$ and $\pi_v^{\iota_\wv^{-1} \Iw(\wv)} \neq 0$, then for every $\sigma \in I_{L_\wv}$, we have
\begin{equation*}\cha_{r_l(\pi)(\sigma)}(X) = \prod_{j=1}^n \left(X - \chi_{v, j}^{-1}(\Art_{L_\wv}^{-1}(\sigma))\right).\end{equation*}
\item If $v \in S_l$ splits as $v = w w^c$ then $r_l(\pi)|_{G_{F_w}}$ is de Rham. If $\pi_w$ is unramified then $r_l(\pi)|_{G_{F_w}}$ is even crystalline. Moreover, for each $\tau \in \widetilde{I}_l$, we have
\begin{equation*}\text{HT}_\tau(r_l(\pi)) = \{ \lambda_{{\tau}, j} + n - j\}_{j=1, \dots, n}.\end{equation*}
\end{enumerate}
Finally if $r_l(\pi)$ is irreducible then for each place $v$ of $L^+$ which splits as $L$ as $w w^c$, the representation $\pi_w \circ \iota_w^{-1}$ is generic.
\end{theorem}
\begin{proof} This follows from Theorem 2.3 of \cite{Gue10} and Corollary 5.10 of \cite{Sha74}.
\end{proof}

\begin{proposition}\label{residualrep} Let $\ffrm$ be a maximal ideal of $\bbT^T_\lambda(U, \cO)$. Then there is a unique continuous semisimple representation
\begin{equation*}\overline{r}_\ffrm : G_L \rightarrow \GL_n(\bbT^T_\lambda(U,\cO)/\ffrm)\end{equation*}
such that:
\begin{enumerate} \item $\overline{r}_\ffrm^c \cong \overline{r}_\ffrm^\vee(1-n).$
\item $\overline{r}_\ffrm$ is unramified outside $T$. For all $v \not\in T$ splitting in $L$ as $v = w w^c$, the characteristic polynomial of $\overline{r}_\ffrm(\Frob_w)$ is
\begin{equation*}X^n + \dots + (-1)^j(\bbN w)^{j(j-1)/2} T_w^j X^{n-j} + \dots + (-1)^n (\bbN w)^{n(n-1)/2} T_w^{n}.\end{equation*}
\end{enumerate}
\end{proposition}
\begin{proof} This follows quickly from the above theorem; see the proof of \cite{Clo08}, Proposition 3.4.2 for details.
\end{proof}
We will fix henceforth a maximal ideal $\ffrm$ such that $\overline{r}_\ffrm$ is absolutely irreducible (or \emph{non-Eisenstein}).

\begin{proposition}\label{heckerep} In this case, $\overline{r}_\ffrm$ admits an extension to a continuous homomorphism
\begin{equation*}\overline{r}_\ffrm : G_{L^+} \rightarrow \cG_n(\bbT_\lambda^T(U, \cO)/\ffrm),\end{equation*}
with the property that $\overline{r}_\ffrm^{-1}(\GL_n \times \GL_1(\bbT_\lambda^T(U, \cO)/\ffrm)) = G_L$ and that $\nu \circ \overline{r}_\ffrm = \epsilon^{1-n} \delta_{L/L^+}^{\mu_\ffrm}$ for some $\mu_\ffrm \in \bbZ/2\bbZ$.
Moreover, $\overline{r}_\ffrm$ admits a continuous lift, unique up to conjugacy, to a representation
\begin{equation*}r_\ffrm : G_{L^+} \rightarrow \cG_n(\bbT_\lambda^T(U, \cO)_\ffrm)\end{equation*}
satisfying the following: 
\begin{enumerate}
\item If $v \not\in T$ is a finite place of $L^+$ which splits as $w w^c$ in $L$, then $r_\ffrm$ is unramified at $w$ and $w^c$, and $r_\ffrm(\Frob_w)$ has characteristic polynomial
\begin{equation*}X^n + \dots + (-1)^j(\bbN w)^{j(j-1)/2} T_w^j X^{n-j} + \dots + (-1)^n (\bbN w)^{n(n-1)/2} T_w^{n}.\end{equation*}
\item If $v$ is a finite place of $L^+$ inert in $L$, and $U_v$ is a hyperspecial maximal compact subgroup of $G(L^+_v)$, then $r_\ffrm$ is unramified at $v$.
\item $\nu \circ r_\ffrm = \epsilon^{1-n} \delta^{\mu_\ffrm}_{L/L^+}$.
\end{enumerate}
\end{proposition}
\begin{proof} This is proved exactly as the first part of Proposition 3.4.4 of \cite{Clo08}. We note that we use here the fact that $l>2$, in the form of Lemma 2.1.12 of \cite{Clo08}.
\end{proof}

\subsection*{A patching argument}

We now suppose $R = \emptyset$, so that $T = S_r \coprod S_l$, for some set $S_r$ of primes. We will define a local deformation problem $\cD_v$ for each $v \in T$, by giving a suitable quotient $R_v$ of the universal lifting ring $R^\square_\wv$ at $\wv$. For $v \in S_r$, take $R_v = R^{\overline{\square}}_\wv$. For $v \in S_l$, we will take the ring $R_\wv^{\lambda,\text{cr}}$. See Section \ref{galoisdef} for the definitions of these rings.

Thus we have the global deformation problem
\begin{equation*}\cS = \left(L/L^+, T, \widetilde{T}, \cO, \overline{r}, \epsilon^{1-n} \delta^{\mu_\ffrm}_{L/L^+}, \{R_\wv^{\overline{\square}}\}_{v \in S_r} \cup \{R_\wv^{\lambda,\text{cr}}\}_{v \in S_l}\right).\end{equation*}
We set $R^\text{loc} =  \left( \widehat{\bigotimes}_{v \in S_r} R_\wv^{\overline{\square}} \right)\widehat{\otimes}_\cO \left(\widehat{\bigotimes}_{v \in S_l} R_\wv^{\lambda,\text{cr}}\right).$

Suppose now that $U=\prod_v U_v$ has the following form:
\begin{enumerate} \item $U_v = G(\cO_{L_v^+})$ if $v \in S_l$.
 \item $U_v = G(\cO_{L_v^+})$ if $v \not\in T$ is split in $L$.
\item $U_v$ is a hyperspecial maximal compact of $G(L_v^+)$ if $v$ is inert in $L$.
\item $U_v$ is arbitrary for $v \in S_r$.
\end{enumerate}
 Suppose also that for all $t \in G(\bbA_{L^+}^\infty)$, the group $t^{-1} G(F^+) t \cap U$ is trivial.
 \begin{theorem}\label{taylorwiles} The lifting
\begin{equation*}r_\ffrm : G_{L^+} \rightarrow \cG_n(\bbT_\lambda^T(U, \cO)_\ffrm)\end{equation*}
defined above is of type $\cS$. Suppose moreover that $\overline{r}_\ffrm(G_{L^+(\zeta_l)})$ is adequate. Let $r : G_{L^+} \rightarrow \cG_n(\cO)$ be another lifting of $\overline{r}_\ffrm$ of type $\cS$. Suppose there exists a homomorphism $f' : \bbT_\lambda^T(U, \cO)_\ffrm \rightarrow \cO$ such that:
\begin{enumerate} \item For $v \in S_l$, $(f' \circ r_\ffrm)|_{G_{L_\wv}}$ and $r|_{G_{L_\wv}}$ lie on the same component of $\Spec R_\wv^{\lambda, \text{cr}} \otimes_\cO \barQl$.
\item For $v \in S_r$, $(f' \circ r_\ffrm)|_{G_{L_\wv}} \leadsto r|_{G_{L_\wv}}.$
\end{enumerate}
Then there exists a second homomorphism $f : \bbT_\lambda^T(U, \cO)_\ffrm \rightarrow \cO$ such that $r$ and $f \circ r_\ffrm$ are conjugate by an element of $\GL_n(\cO)$.

\end{theorem}

\begin{proof}
That $r_\ffrm$ is of type $\cS$ follows immediately from its construction and the fact that $\bbT^T_\lambda(U, \cO)_\ffrm$ is reduced and $l$-torsion free. Let
\begin{equation*}q_0 = [L^+ : \bbQ]n(n-1)/2 + [L^+ : \bbQ]n(1- (-1)^{{\mu_\ffrm} -n})/2.\end{equation*}
By Proposition \ref{Chebotarev}, we may choose for every $N \geq 1$ a tuple \[(Q_N, \widetilde{Q}_N, \{\psi_{\wv}\}_{v \in Q_N})\] such that \begin{itemize}
\item $\#Q_N = q$ for all $N$.
\item $\bbN v \equiv 1 \mod l^N$ for all $v \in Q_N$.
\item The ring $R^{\square_T}_{\cS_{Q_N}}$ can be topologically generated over $R^\text{loc}$ by
\begin{equation*}g = q-q_0 = q - [L^+ : \bbQ]n(n-1)/2 + [L^+ : \bbQ]n(1- (-1)^{\mu_\ffrm -n})/2\end{equation*}
elements.
\item $\psi_{\wv}$ is a subrepresentation of $\overline{r}|_{G_{L_\wv}}$ of dimension $d_N^v$.
\end{itemize}
For each $v, N \geq 1$, let $\frp_N^\wv,\frp_{N,1}^\wv$ be the parahoric subgroup of $\GL_n(\cO_{L_\wv})$ corresponding to the partition $n = (n - d_N^v) + d_N^v$, as defined previously. We let
\begin{equation*}U_1(Q_N) = \prod_v U_1(Q_N)_v, U_0(Q_N) = \prod_v U_0(Q_N)_v\end{equation*}
be the compact open subgroups of $G(\bbA^\infty_{L^+})$ with $U_i(Q_N) = U_v$ for $v \not\in Q_N$, $U_1(Q_N) =  \iota_{\wv}^{-1} \frp_{N,1}^\wv$, $U_0(Q_N) = \iota_{\wv}^{-1} \frp_{N}^\wv$ for $v \in Q_N$. 

We have natural maps
\[ \xymatrix@1{\bbT_\lambda^{T\cup Q_N}(U_1(Q_N), \cO) \ar[r] & \bbT_\lambda^{T \cup Q_N}(U_0(Q_N), \cO) \ar[r] &}\]
\[ \xymatrix@1{ \ar[r] & \bbT_\lambda^{T \cup Q_N}(U, \cO) \ar[r] &\bbT_\lambda^T(U, \cO).} \]
The first two arrows are surjective, the last is injective. Thus $\ffrm$ determines maximal ideals denoted $\ffrm_{Q_N}, \ffrm_{Q_N}, \ffrm, \ffrm$ of these four algebras.
After localizing at $\ffrm$ the last map is an isomorphism, since $\overline{r}_\ffrm$ is absolutely irreducible (cf. the proof of \cite{Clo08}, Corollary 3.4.5).

For each $v \in Q_N$, let $\phi_\wv \in G_{L_\wv}$ be a Frobenius lift and let $\varpi_\wv$ be the uniformizer with $\Art_{L_\wv}\varpi_\wv = \phi_\wv |_{L_\wv^\text{ab}}$. In Proposition \ref{mainprop} we have defined commuting projection operators $\pr_{\varpi_\wv}$ at each place $\wv$. For $i=0, 1$, we set
\begin{equation*}H_{i, Q_N} = \left(\prod_{v \in Q_N} \pr_{\varpi_\wv} \right)S_\lambda(U_i(Q_N), \cO)_{\ffrm_{Q_N}},\end{equation*}
and let $\bbT_{i, Q_N}$ denote the image of $\bbT_\lambda^{T \cup Q_N}(U_i(Q_N), \cO)$ in $\End_\cO(H_{i, Q_N}).$ Let $H = S_\lambda(U, \cO)_\ffrm$. 

Let $\Delta_{Q_N} = U_0(Q_N)/U_1(Q_N)$. Let $\fra_{Q_N} \subset \cO[\Delta_{Q_N}]$ denote the augmentation ideal. For each $\alpha \in \cO_{L_\wv}^\times$, let $V_\alpha$ denote the Hecke operator 
\begin{equation*}V_\alpha = \iota_\wv^{-1}\left(\left[\frp_{N,1}^\wv\left(\begin{array}{cc} 1_{n - d_N^v} & 0 \\ 0 & A_\alpha \end{array} \right)\frp_{N,1}^\wv\right]\right),\end{equation*}
where $A_\alpha=\diag(\alpha, 1, \dots, 1).$

We claim that the following hold:
\begin{enumerate} \item For each $N$, the map
\begin{equation*}\prod_{v \in Q_N} \pr_{\varpi_\wv} : H \rightarrow H_{0, Q_N}\end{equation*}
is an isomorphism.
\item For each $N$, $H_{1, Q_N}$ is free over $\cO[\Delta_{Q_N}]$ with
\begin{equation*}H_{1, Q_N}/\fra_{Q_N} \overset{\sim}{\rightarrow} H_{0, Q_N}.\end{equation*}
\item For each $N$ and for each $v \in Q_N$, there is a character $V_\wv : \cO_{L_\wv}^\times \rightarrow \bbT_{1, Q_N}^\times$ such that
\begin{enumerate} \item for any $\alpha$, $V_\alpha = V_\wv(\alpha)$ on $H_{1, Q_N}$.
\item ($r_{\ffrm_{Q_N}} \otimes \bbT_{1, Q_N})|_{G_{L_\wv}} \cong s \oplus \psi$ with $s$ unramified, lifting $\overline{s}_\wv$, $\psi$ lifting $\overline{\psi}_\wv$, and $\psi|_{I_{L_\wv}}$ acting as the scalar character $V_\wv \circ \Art^{-1}_{L_\wv}|_{I_{L_{\wv}}}$.
\end{enumerate}
\end{enumerate}
In fact these follow respectively from Proposition \ref{mainprop}, Lemma \ref{freeness} and Proposition \ref{hecke}.

For each $N$, the lift $r_{\ffrm_{Q_N}} \otimes \bbT_{1, Q_N}$ is of type $\cS_{Q_N}$ and gives rise to a surjection $R^\text{univ}_{\cS_{Q_N}} \rightarrow \bbT_{1,Q_N}$. Thinking of $\Delta_{Q_N}$ as the maximal $l$-power order quotient of $\prod_{v \in Q_N} I_{L_\wv}$, we obtain a homomorphism $\Delta_{Q_N} \rightarrow (R^\text{univ}_{\cS_{Q_N}})^\times$ as follows: for each $v \in Q_N$, decompose $r^\text{univ}_{\cS_{Q_N}}|_{G_{L_\wv}} = s \oplus \psi$ and take a diagonal entry of $\psi$ in some basis. This does not depend on the choice of basis.

We thus have homomorphisms $\cO[\Delta_{Q_N}] \rightarrow R^\text{univ}_{\cS_{Q_N}} \rightarrow R^{\square_T}_{\cS_{Q_N}}$ and natural isomorphisms $R^\text{univ}_{\cS_{Q_N}}/\fra_{Q_N} \cong R^\text{univ}_\cS$ and $R^{\square_T}_{\cS_{Q_N}}/\fra_{Q_N} \cong R^{\square_T}_{\cS}$. Moreover, this makes the maps $R^\text{univ}_{\cS_{Q_N}} \rightarrow \bbT_{1,Q_N}$ into homomorphisms of $\cO[\Delta_{Q_N}]$-algebras.

At this point we can apply Lemma \ref{patching} below, with the following identifications: we take 
\begin{equation*}T = \bbT^T_\lambda(U, \cO), H=S_\lambda(U, \cO)_\ffrm.\end{equation*}
\begin{equation*}R = R^\text{univ}_\cS, R_N = R^\text{univ}_{\cS_{Q_N}}.\end{equation*}
\begin{equation*}t = \#T,\text{ and }q, g \text{ are as constructed above}.\end{equation*}
We choose a lifting $r^\text{univ}_\cS : G_{L^+} \rightarrow \cG_n(R^\text{univ}_\cS)$ representing the universal deformation, and for every $N$ a lifting $r^\text{univ}_{\cS_{Q_N}} : G_{L^+} \rightarrow \cG_n(R^\text{univ}_{\cS_{Q_N}})$ reducing to the this one. This induces compatible isomorphisms
\begin{equation*}R^{\square_T}_{\cS} \rightarrow R^\text{univ}_\cS \widehat{\otimes}_\cO \cT\text{ and }R^{\square_T}_{\cS_{Q_N}} \rightarrow R^\text{univ}_{\cS_{Q_N}} \widehat{\otimes}_\cO \cT.\end{equation*}
The rings $R^{\square_T}_{\cS}$ and $R^{\square_T}_{\cS_{Q_N}}$ are naturally $R^\text{loc}$-algebras, and we take $R^L = R^\text{loc}$. We take 
\begin{equation*}S_\infty = \cO[[x_1, \dots, x_{n^2t}, y_1, \dots, y_q]],\end{equation*}
 and let $\fra$ be the kernel of the augmentation map $S_\infty \rightarrow \cO$. Finally applying the lemma, we are given a $R_\infty$-module $H_\infty$ satisfying the following:

\begin{enumerate} \item $H_\infty$ receives an action of $S_\infty$, which commutes with the $R_\infty$ action, and $H_\infty/\fra H_\infty \cong H$, compatibly with the homomorphism $R^\text{loc} \rightarrow R^{\square_T}_{\cS} \rightarrow R^\text{univ}_\cS$.
\item $H_\infty$ is a finite free $S_\infty$-module.
\item The action of $S_\infty$ on $H_\infty$ factors through the action of $R_\infty$.
\end{enumerate}

In particular, we have that $\depth_{R_\infty} H_\infty = \dim S_\infty = 1 + n^2 \#T + q$. On the other hand, $R_\infty$ is equidimensional of dimension $1 + n^2 \#T + n(n-1)/2[L^+ : \bbQ] + g.$ It follows that
\begin{equation*}1 + n^2 \#T + q \leq 1 + n^2 \#T + q - [L^+ : \bbQ]n(1 - (-1)^{n-\mu_\ffrm})/2. \end{equation*}
Thus equality holds, with $n \equiv \mu_\ffrm$ mod $2$. In particular, the support of $H_\infty$ on $R_\infty$ is a union of irreducible components (cf. Lemma 2.3 of \cite{Tay08}).

For each $v \in T$, let $\cC_v$ be the unique component of $R^{\overline{\square}}_\wv$ or $R_\wv^{\lambda,\text{cr}}$ containing $(f' \circ r_\ffrm)|_{G_{L_\wv}}.$ Define a new deformation problem
\begin{equation*}\cS' = \left(L/L^+, T, \widetilde{T}, \cO, \overline{r}, \epsilon^{1-n} \delta^{\mu_\ffrm}_{L/L^+}, \{R^{\cC_v}_\wv\}_{v \in S_r} \cup \{R_\wv^{\lambda, \cC_v}\}_{v \in S_l}\right).\end{equation*}
There is a natural surjection $R^\text{univ}_\cS \rightarrow R^\text{univ}_{\cS'}$.

The hypotheses of the theorem imply that the closed point of $R_\infty[1/l]$ induced by $f'$ is contained in a unique irreducible component, hence the entirety of this irreducible component is contained in the support of $H_\infty$. But the subset $\Spec R^\text{univ}_{\cS'} \subset \Spec R_\infty$ is contained in this component. In particular,
\begin{equation*}\Supp_{R^\text{univ}_{\cS'}} H_\infty \otimes_{R_\infty} R^\text{univ}_{\cS'} = \Spec R^\text{univ}_{\cS'}.\end{equation*}
Thus $H \otimes_{R^\text{univ}_\cS} R^\text{univ}_{\cS'}$ is a nearly faithful $R^\text{univ}_{\cS'}$-module, and the homomorphism $R^\text{univ}_{\cS'} \rightarrow \cO$ induced by $r$ actually factors through $\bbT_\lambda^T(U, \cO)_\ffrm$. This concludes the proof of the theorem.
\end{proof}

\begin{corollary}\label{minimalfinite} We have $\mu_\ffrm \equiv n \mod 2$. Moreover, let $\cS'$ be the deformation problem defined in the proof above. Then with the hypotheses of the theorem, $R^\text{univ}_{\cS'}$ is a finite $\cO$-algebra.
\end{corollary}
\begin{proof} The proof actually shows that $R^\text{univ}_{\cS'}/J$ is a quotient of the finite $\cO$-algebra $\bbT_\lambda^T(U, \cO)_\ffrm$, for some nilpotent ideal $J$ of $R^\text{univ}_{\cS'}$. Then $R^\text{univ}_{\cS'}/\lambda$ is an Artinian $k$-algebra, hence of finite length; now apply \cite{Mat07}, Theorem 8.4.
\end{proof}

We suppose that all rings in the statement of the following lemma are Noetherian.

\begin{lemma}\label{patching} Let $R \rightarrow T$ be a surjective $\cO$-algebra homomorphism, with $T$ a finite $\cO$-algebra. Let $H$ be a finite $T$-module. Fix positive integers $t, q$. Suppose that we have for each $N \geq 1$ a surjective homomorphism $R_N \rightarrow T_N$ of $S_N = \cO[(\bbZ/l^N\bbZ)^q]$ algebras which reduces to the homomorphism $R \rightarrow T$ on quotienting out by the augmentation ideal. Let $H_N$ be a finite $T_N$-module which is free of finite rank as an $S_N$-module. Suppose that this rank is independent of $N$.

Let $\cT = \cO[[x_1, \dots, x_{n^2 t}]]$, and let $S_\infty = \cT[[y_1, \dots, y_q]]$. Let $R^L$ be an $\cO$-algebra and suppose that $R_N \widehat{\otimes}_\cO \cT, R \widehat{\otimes}_\cO \cT$ are $R^L$ algebras in such a way that $R_N \widehat{\otimes} \cT \rightarrow R \widehat{\otimes}_\cO \cT$ is a homomorphism of $R^L$-algebras. Finally suppose that each $R_N \widehat{\otimes}_\cO \cT$ can be topologically generated over $R^L$ by $g$ generators. Let $R_\infty = R^L[[x_1, \dots, x_g]]$. Let $\fra$ be the ideal of $S_\infty$ generated by the $x_i$ and $y_j$. Then:
\begin{enumerate} \item There exists an $R_\infty$-module $H_\infty$ which receives a commuting action of $S_\infty$, and $H_\infty/\fra H_\infty \cong H$, compatibly with the homomorphism $R^L \rightarrow R \widehat{\otimes}_\cO \cT \rightarrow R$.
\item $H_\infty$ is a finite free $S_\infty$-module.
\item The action of $S_\infty$ on $H_\infty$ factors through the action of $R_\infty$.
\end{enumerate}
\end{lemma}
\begin{proof}
We do not prove this here but note that it can easily be extracted from e.g. the proof of Theorem 3.6.1  of \cite{Bar10}.
\end{proof}
\section{Minimal automorphy lifting}

 \begin{theorem}\label{minimal} Let $F$ be an imaginary CM field with totally real subfield $F^+$ and let $c$ be the non-trivial element of $\Gal(F/F^+)$. Let $n \in \bbZ_{\geq 1}$ and let $l$ be an odd prime. Let $K \subset \overline{\bbQ}_l$ denote a finite extension of $\bbQ_l$ with ring of integers $\cO$, residue field $k$, and maximal ideal $\lambda$. Let
 \begin{equation*} \rho : G_F \rightarrow \GL_n(\cO)\end{equation*}
 be a continuous representation and let $\overline{\rho} = \rho \mod \lambda$. Let $\mu : G_{F^+} \rightarrow \cO^\times$ be a continuous character. Suppose that $(\rho, \mu)$ satisfy the following properties:
  \begin{enumerate} \item $\rho^c \cong \rho^\vee \epsilon^{1-n} \mu|_{G_F}.$
  \item $\mu(c_v)$ is independent of $v \mid \infty$.
  \item $\rho$ is ramified at only finitely many places.
 \item $\overline{\rho}$ is absolutely irreducible and $\overline{\rho}(G_{F(\zeta_l)}) \subset \GL_n(k)$ is adequate.
 \item $\zeta_l \not\in F$.
 \item There is a continuous representation $\rho' : G_F \rightarrow \GL_n(\cO)$, a continuous character $\mu' : G_{F^+} \rightarrow \cO^\times$, a RAECSDC automorphic representation $(\pi, \chi)$ of $\GL_n(\bbA_F)$ which is potentially unramified above $l$ and $\iota : \overline{\bbQ}_l \overset{\sim}{\rightarrow} \bbC$ such that 
 \begin{enumerate} 
 \item $\rho' \otimes_\cO \barQl \cong r_{l, \iota}(\pi) : G_F \rightarrow \GL_n(\barQl).$
 \item $\mu' \otimes_\cO \barQl \cong r_{l, \iota}(\chi).$
 \item $(\overline{\rho},\overline{\mu}) = (\overline{\rho}',\overline{\mu}')$.
 \item For all places $v \nmid l$ of $F$, either \begin{itemize} \item $\pi_v$ and $\rho|_{G_{F_v}}$ are both unramified, or \item $\rho'|_{G_{F_v}} \leadsto \rho|_{G_{F_v}}.$
 \end{itemize}
 \item For all places $v \mid l, \rho'|_{G_{F_v}} \sim \rho |_{G_{F_v}}$.
 \end{enumerate}
 \end{enumerate}
 Then $(\rho, \mu)$ is automorphic. Moreover, if $\pi$ has level prime to $l$ and $\rho$ is crystalline, then $(\rho, \mu)$ is automorphic of level prime to $l$.
 \end{theorem} 
 \begin{proof} We use twisting and soluble base change to reduce the theorem to Theorem \ref{taylorwiles}, proved above. Let $C$ be the set of places of $F^+$ above which $\rho$ or $\pi$ is ramified, together with the places of $F^+$ dividing $l$. After a soluble base change we can assume that each place of $C$ splits in $F$ and that $\rho$ is crystalline and $\pi$ has level prime to $l$. The final sentence of the theorem will follow from the observation that in the soluble extensions $L/F$ below, we can take all primes of $F$ above $l$ to split completely in $L$.
  
Let us suppose first that $\mu = \chi = \delta_{F/F^+}^n$. 
Let $L/F$ be an imaginary CM extension such that:
\begin{enumerate} \item $L/F$ is soluble.
\item $L$ is linearly disjoint with $\overline{F}^{\ker \overline{\rho}}(\zeta_l)$ over $F$.
\item 4 divides $[L^+:F^+]$.
\item $L/L^+$ is unramified at all finite places.
\end{enumerate}

In particular, the hypotheses of Section \ref{twk} are satisfied. We write $G_{/\cO_{L^+}}$ for the algebraic group constructed there. By Th\'eor\`eme 5.4 and Corollaire 5.3 of \cite{Lab09}, there exists an automorphic representation $\Pi$ of $G(\bbA_{L^+})$ such that $\pi_L$ is a strong base change of $\Pi$, in the terminology of that paper.

Let $v_1$ be a place of $L^+$ split in $L$, of residue characteristic not dividing the order of any element of $G(L^+)$, and above which both $\rho|_{G_{L}}$ and $\pi_L$ are unramified.  For a place $w$ of $L$ above $v_1$, we have $\rho' |_{G_{F_w}} \leadsto \rho|_{G_{F_w}}$ (since $\pi_w$ is generic).

We let $S_r$ be the set of primes of $L^+$ above which $\pi_L$ or $\rho|_{G_L}$ are ramified, along with the prime $v_1$. Let $S_l$ be the set of places of $L^+$ dividing $l$. Let $T = S_l \cup S_r$, and choose for every $v \in T$ a place $\wv$ of $L$ above $v$. Denote the set of such places $\widetilde{T}$.

Choose an open compact subgroup $U = \prod_v U_v$ of $G(\bbA^\infty_{L^+})$ satisfying the following: 
\begin{enumerate}\item $U_v = G(\cO_{L_v^+})$ if $v \in S_l$.
 \item $U_v = G(\cO_{L_v^+})$ if $v \not\in T$ is split in $L$.
\item $U_v$ is a hyperspecial maximal compact of $G(L_v^+)$ if $v$ is inert in $L$.
\item $U_{v_1} = \iota_{\wv_1}^{-1}\Iw_1(\wv_1)$.
\item $\Pi_v^{U_v} \neq 0$ for $v \in S_r$, $v \neq v_1$.
\end{enumerate}
Then for all $t \in G(\bbA_{L^+}^\infty)$, the group $t^{-1} G(L^+) t \cap U$ is trivial (the addition of the place $v_1$ ensures this). Let $\widetilde{I}_l$ denote the set of embeddings $L \hookrightarrow K$ giving rise to an element of $\widetilde{T}$. For $\tau \in \widetilde{I}_l$, giving rise to $\wv$, let $\lambda_\tau \in \bbZ^n_+$ be such that $HT_\tau(\rho'|_{G_{L_\wv}}) = \{ \lambda_{\tau, j} + n - j \}_{j=1, \dots, n}.$ Then $\lambda \in (\bbZ^n_+)^{\widetilde{I}_l}$ defines a space of automorphic forms $S_\lambda(U, \cO)$, as in the previous section.

The action of the Hecke operators $T^j_w$ on the space $(\iota^{-1}\Pi^\infty)^U$ defines a homomorphism $\bbT^T_\lambda(U, \cO) \rightarrow \barQl$. After extending $K$, we can suppose that this homomorphism is valued in $\cO$. Let $\ffrm$ be the maximal ideal containing the kernel. 

In the notation of the previous section, $\ffrm$ is non-Eisenstein, so the representation $\overline{\rho}'|_{G_L}$ admits an extension to a continuous representation
\begin{equation*}\overline{r}_\ffrm : G_{L^+} \rightarrow \cG_n(k).\end{equation*}
Since $\zeta_l \not\in L$, $\overline{r}_\ffrm(G_{L^+(\zeta_l)})$ is adequate. After possibly extending $K$ again, $\rho|_{G_L}$ admits an extension to a continuous representation
\begin{equation*}r : G_{L^+} \rightarrow \cG_n(\cO)\end{equation*}
lifting $\overline{r}_\ffrm$. In particular, $r$ is a lifting of type $\cS$. Theorem \ref{taylorwiles} now applies, and then Lemma \ref{basechange} shows that $\rho$ is automorphic.

We now suppose that $\mu, \chi$ are arbitrary and reduce to the case treated above. By Lemma 4.1.4 of \cite{Clo08}, we can find a character $\psi : \bbA_F^\times/F^\times \rightarrow \bbC^\times$ of type $A_0$ such that
\begin{equation*}\psi \circ \bbN_{F/F^+} = \chi \circ \bbN_{F/F^+},\end{equation*}
with $\psi$ unramified above $l$ and at all primes $v$ where $\pi$ and $\rho$ are unramified. Twisting both $\pi$ and $\rho$, we can therefore assume that $\chi = \delta_{F/F^+}^n$ and $\overline{\mu} = \delta^n_{F/F^+}$. (It is easy to see that this preserves the hypotheses (d) and (e) above). Now let $S$ be the set of primes of $F^+$ dividing $l$ or at which $\pi$ or $\rho$ are ramified. For each place $v \in S$, choose a place $\wv$ of $F$ lying above it. 

Applying Lemma 4.1.6 of \cite{Clo08}, we can find a character $\theta : G_F \rightarrow \barQl^\times$ such that $\overline{\theta}$ is trivial, $\mu|_{G_F} = \theta \theta^c$, and such that for each $v \in S$, $\theta$ is unramified at $\wv^c$. Since $\mu$ is crystalline, $\theta$ is crystalline at each prime above $l$. Making a further soluble base change, we can suppose that $\theta$ is also unramified outside $S$.

We now twist $\rho$ by $\theta^{-1}$ to reduce to the case where $\mu = \delta^n_{F/F+}$, treated above. Thus the proof of the theorem will be complete as soon as we show that this preserves the hypotheses (d) and (e) above. For $v \in S$, we have
\begin{equation*}\rho'|_{G_{F_{\wv^c}}} \leadsto (\rho' \otimes \theta^{-1})|_{G_{F_{\wv^c}}} \leadsto (\rho \otimes \theta^{-1})|_{G_{F_{\wv^c}}}.\end{equation*}
(Use that $\theta$ is unramified at $\wv^c$ and apply Lemma 3.4.2 of \cite{Bar10}). Then we have
\begin{equation*}\rho'|_{G_{F_{\wv}}} \cong (\rho'^\vee)|^c_{G_{F_{\wv^c}}} \otimes \epsilon^{1-n} \leadsto (\rho\otimes \theta^{-1})^\vee|^c_{G_{F_{\wv^c}}} \otimes \epsilon^{1-n} \cong (\rho \otimes \theta^{-1})|_{G_{F_{\wv}}}.\end{equation*} 
The theorem follows.
\end{proof} 
\begin{remark} In many cases the hypothesis (d) can be weakened to the following: for each $v \nmid l$, either
\begin{itemize} \item $\rho|_{G_{F_v}}$ and $\pi_v$ are both unramified, or
\item $\rho'|_{G_{F_{v}}} \sim \rho|_{G_{F_v}}.$
\end{itemize}
We explain this remark as follows. If $\rho_1 \sim \rho_2$ and $\iota WD(\rho_1)^{F-ss} = \rec_K(\pi)$ for some generic irreducible smooth representation $\pi$ of $\GL_n(M)$ over $\bbC$, then in fact $\rho_1 \leadsto \rho_2$, as one can check that the generic fibre of the local lifting ring is formally smooth at $\rho_1$ in this case.  Thus one can use the weaker hypothesis whenever local-global compatibility is known to hold in this form for $\pi$, at the place $v$. At the time of writing this is known to be true, for example, when $\pi_v$ becomes unramified after a finite base change (cf. \cite{Che09}) or when $\pi$ has slightly regular weight in the sense of \cite{Shi09}.
\end{remark}

\section{Ordinary forms on definite unitary groups}

We take up the assumptions of the beginning of Section \ref{twk}. Thus $L/L^+$ is an imaginary $CM$ field and $G$ is a unitary group over $\cO_{L^+}$, $S_l$ is the set of places of $L^+$ dividing $l$, and $R$ is a set of places disjoint from $L^+$. We suppose that $T$ is a set of places of $L^+$ containing $R \cup S_l$, and that every place in $T$ is split in $L$. We fix an open compact subgroup $U = \prod_v U_v$.

We are going to follow \cite{Ger09} in recalling some constructions in Hida theory and proving an $R=\bbT$ type result in the ordinary case. As such we will briefly sketch the theory from \emph{op. cit.}, and refer the reader to that paper for a more complete development. For integers $0 \leq b \leq c$, and $v \in S_l$, we consider the subgroup $\Iw(\wv^{b, c}) \subset \GL_n(\cO_{L_\wv})$ defined as those matrices which are congruent to an upper-triangular matrix modulo $\wv^c$ and congruent to a unipotent upper-triangular matrix modulo $\wv^b$. We set $U(\frl^{b, c}) = U^l \times \prod_{v \in S_l} \Iw(\wv^{b, c})$. 

We define some additional Hecke operators at the places dividing $l$ on the spaces $S_{\lambda, \{\chi\}}(U(\frl^{b,c}), A).$ Let $v$ be a place of $L^+$ dividing $l$, and let $\varpi_\wv$ be a uniformizer of $L_\wv$. We have an $n \times n$ matrix
\begin{equation*}\alpha^j_{\varpi_\wv} = \diag(\underbrace{\varpi_\wv, \dots, \varpi_\wv}_j, 1, \dots, 1).\end{equation*}
We set
\begin{equation*}U^j_{\lambda, \varpi_\wv} = (w_0 \lambda)(\alpha^j_{\varpi_\wv})^{-1} \left[U(\frl^{b,c}) \iota_\wv^{-1} (\alpha^j_{\varpi_\wv}) U(\frl^{b,c})\right].\end{equation*}
We refer to \cite{Ger09}, Section 2.2 for the definition of the character $w_0 \lambda$.

If $u \in T(\cO_{L_\wv})$ then we write 
\begin{equation*}\langle u \rangle = \left[ U(\frl^{b,c}) \iota_\wv^{-1} (u) U(\frl^{b,c})\right].\end{equation*}
It is proved in \cite{Ger09} that for any $\cO$-module $A$, these operators act on the spaces $S_{\lambda, \{\chi\}}(U(\frl^{b,c}), A)$ and commute with the inclusions
\begin{equation*}S_{\lambda, \{\chi\}}(U(\frl^{b,c}), \cO) \subset S_{\lambda, \{\chi\}}(U(\frl^{b',c'}), \cO),\end{equation*}
when $b \leq b'$ and $c \leq c'$. We write $\bbT^T_{\lambda, \{\chi_v\}}(U(\frl^{b, c}), A)$ for the $\cO$-subalgebra of $\End_\cO(S_{\lambda, \{\chi_v\}}(U, A))$ generated by the operators $T^j_w$ and $(T^n_w)^{-1}$ as above and all the operators $\langle u \rangle$ for 
\begin{equation*}u \in T(\cO_{L^+, l}) = \prod_{v \in S_l} T(\cO_{L^+_v}).\end{equation*}
With these identifications, the operators $\langle u \rangle$ endow each Hecke algebra $\bbT^T_{\lambda, \{\chi_v\}}(U(\frl^{b, c}), A)$ with the structure of algebra for the completed group ring
\begin{equation*}\Lambda^+ = \cO[[T(\cO_{L^+, l})]]\end{equation*}
and for its subring
\begin{equation*}\Lambda = \cO[[T(\frl)]],\end{equation*}
where $T(\frl)$ is defined by the exact sequence
\begin{equation*}\xymatrix{ 0 \ar[r] & T(\frl) \ar[r] & \prod_{v \in S_l} T(\cO_{L^+_v}) \ar[r] & \prod_{v \in S_l} k(v)^\times \ar[r] & 0.}\end{equation*}

We have the ordinary idempotent $e = \lim_{r\rightarrow \infty} U(\frl)^{r!}$, where we set
\begin{equation*}U(\frl) = \prod_{v \in S_l} \prod_{j=1}^n U^j_{\lambda, \varpi_\wv}.\end{equation*}
\begin{definition} We define the ordinary Hecke algebra 
\begin{equation*}\bbT^{T,\text{ord}}_{\lambda, \{\chi_v\}}(U(\frl^{b,c}), A) = e\bbT^T_{\lambda, \{\chi_v\}}(U(\frl^{b,c}), A).\end{equation*}
Equivalently, $\bbT^{T,\text{ord}}_{\lambda, \{\chi_v\}}(U(\frl^{b,c}), A) $ is the image of $\bbT^T_{\lambda, \{\chi_v\}}(U(\frl^{b,c}), A)$ in the algebra of $\cO$-endomorphisms of the space
\begin{equation*}S^\text{ord}_{\lambda, \{\chi_v\}}(U(\frl^{b,c}), A)=eS_{\lambda, \{\chi_v\}}(U(\frl^{b,c}), A).\end{equation*}
\end{definition}

We also consider the space 
\begin{equation*}S_{\lambda, \{ \chi_v \}}(U(\frl^\infty), K/\cO) = \ilim_c S_{\lambda, \{\chi_v\}}(U(\frl^{c, c}), K/\cO),\end{equation*}
which receives a faithful action of the algebra
\begin{equation*}\bbT^T_{\lambda, \{\chi_v\}}(U(\frl^\infty), K/\cO) = \plim_c \bbT^T_{\lambda, \{\chi_v\}}(U(\frl^{c,c}), K/\cO).\end{equation*}
We recall (\cite{Ger09}, Lemma 2.4.7) that this algebra is naturally isomorphic to
\begin{equation*}\bbT^T_{\lambda, \{\chi_v\}}(U(\frl^\infty), \cO) = \plim_c \bbT^T_{\lambda, \{\chi_v\}}(U(\frl^{c,c}), \cO).\end{equation*}
Finally we can apply the idempotent $e$ to these spaces, in which case we decorate them with `ord' superscripts.

\begin{proposition}\label{free}  Suppose that for all $t \in G(\bbA_{L^+}^\infty)$, the group 
\begin{equation*}t^{-1} G(L^+) t \cap U\end{equation*}
 contains no element of order $l$. Then $S^\text{ord}_{\lambda, \{\chi_v\}}(U(\frl^\infty), K/\cO)^\vee$ is a free $\Lambda$-module of rank
\begin{equation*}r = \dim_kS^\text{ord}_{\lambda, \{ \chi_v \}}(U(\frl^{1,1}), \bbF).\end{equation*}
\end{proposition}
\begin{proof} The proof is the same as that of Proposition 2.5.3 of \cite{Ger09}, but references to Lemma 2.2.6 of that paper should be replaced with references to Lemma \ref{suffsmall} above.
\end{proof}
The following is Definition 2.6.2 of \cite{Ger09}.
\begin{definition}  We define a homomorphism $T(\frl) \rightarrow \bbT^{T, \text{ord}}_{0, \{\chi_v\}}(U(\frl^\infty), \cO)^\times$
 by 
 \begin{equation*}u \mapsto \left( \prod_{\tau \in \widetilde{I}_l} \prod_{i=1}^n \tau(u_i)^{1-i}\right) \langle u \rangle.\end{equation*}
  This gives rise to an $\cO$-algebra homomorphism $\Lambda \rightarrow \bbT^{T, \text{ord}}_{0, \{\chi_v\}}(U(\frl^\infty), \cO)$, and we write
\begin{equation*}\bbT^{T, \text{ord}}_{\{\chi_v\}}(U(\frl^\infty), \cO)\end{equation*}
for the $\bbT^{T, \text{ord}}_{0, \{\chi_v\}}(U(\frl^\infty), \cO)$, endowed with this $\Lambda$-algebra structure. 
This is the \emph{universal ordinary Hecke algebra of level $U$}.
\end{definition}

\subsection*{Galois representations}
The following proposition is proved in exactly the same way as Proposition \ref{residualrep}.
\begin{proposition} Let $\ffrm$ be a maximal ideal of $\bbT^{T, \text{ord}}_{0, \{\chi_v\}}(U(\frl^\infty), \cO)$. Then there is a unique continuous semisimple representation
\begin{equation*}\overline{r}_\ffrm : G_L \rightarrow \GL_n(\bbT^{T, \text{ord}}_{ \{\chi_v\}}(U(\frl^\infty), \cO)/\ffrm)\end{equation*}
such that:
\begin{enumerate} \item $\overline{r}_\ffrm^c \cong \overline{r}_\ffrm^\vee(1-n).$
\item $\overline{r}_\ffrm$ is unramified outside $T$. For all $v \not\in T$ splitting in $L$ as $v = w w^c$, the characteristic polynomial of $\overline{r}_\ffrm(\Frob_w)$ is
\begin{equation*}X^n + \dots + (-1)^j(\bbN w)^{j(j-1)/2} T_w^j X^{n-j} + \dots + (-1)^n (\bbN w)^{n(n-1)/2} T_w^{n}.\end{equation*}
\end{enumerate}
\end{proposition}
We will fix henceforth a maximal ideal $\ffrm$ such that $\overline{r}_\ffrm$ is absolutely irreducible. The following is proved in exactly the same way as Proposition \ref{heckerep}.
\begin{proposition} In this case $\overline{r}_\ffrm$ admits an extension to a continuous homomorphism
\begin{equation*}\overline{r}_\ffrm : G_{L^+} \rightarrow \cG_n(\bbT^{T, \text{ord}}_{\{\chi_v\}}(U(\frl^\infty), \cO)/\ffrm),\end{equation*}
with the property that $\overline{r}_\ffrm^{-1}(\GL_n \times \GL_1(\bbT^{T, \text{ord}}_{ \{\chi_v\}}(U(\frl^\infty), \cO)/\ffrm)) = G_L$ and $\nu \circ \overline{r} = \epsilon^{1-n} \delta^{\mu_\ffrm}_{L/L^+}$, for some $\mu_\ffrm \in \bbZ/2\bbZ$.
Moreover, $\overline{r}_\ffrm$ admits a continuous lift, unique up to conjugacy, to a representation
\begin{equation*}r_\ffrm : G_{L^+} \rightarrow \cG_n(\bbT^{T, \text{ord}}_{\{\chi_v\}}(U(\frl^\infty), \cO)_\ffrm)\end{equation*}
satisfying the following: 
\begin{enumerate}
\item If $v \not\in T$ is a finite place of $L^+$ which splits as $w w^c$ in $L$, then $r_\ffrm$ is unramified at $w$ and $w^c$, and $r_\ffrm(\Frob_w)$ has characteristic polynomial
\begin{equation*}X^n + \dots + (-1)^j(\bbN w)^{j(j-1)/2} T_w^j X^{n-j} + \dots + (-1)^n (\bbN w)^{n(n-1)/2} T_w^{n}.\end{equation*}
\item If $v$ is a finite place of $L^+$ inert in $L$, and $U_v$ is a hyperspecial maximal compact subgroup of $G(L^+_v)$, then $r_\ffrm$ is unramified at $v$.
\item Suppose that if $v \in R$ then $U_v = \iota_\wv^{-1} \Iw(\wv)$. Then for every $\sigma \in I_{L_\wv}$, we have
\begin{equation*}\cha_{r_\ffrm(\sigma)}(X) = \prod_{j=1}^n \left(X - \chi_{v, j}^{-1}(\Art_{L_\wv}^{-1}(\sigma))\right).\end{equation*}
\item $\nu \circ r_\ffrm = \epsilon^{1-n} \delta^{\mu_\ffrm}_{L/L^+}$.
\end{enumerate}
\end{proposition}

\subsection*{Another patching argument}

We specialize to the case where $T = R \cup S_l \cup \{v_1\}$, for some place $v_1$ of $L^+$ split in $L$. We will assume the following:
\begin{enumerate} \item For each each $v \in R \cup S_l$, $\overline{r}_\ffrm(G_{L_\wv})$ is trivial.
\item For each $v \in R$, $\bbN v \equiv 1 \mod l$, and that if $l^N \| (\bbN v - 1)$ then $l^N > n$ and $\cO$ contains an $(l^N)^\text{th}$ root of unity.
\item $\overline{r}_\ffrm$ is unramified above $v_1$, and $\bbN v_1 \not\equiv 1 \mod l$. 
\item For each $v \in R$, the characters $\chi_{v,1}, \dots, \chi_{v,n} : \cO^\times_{L_\wv} \rightarrow \cO^\times$ become trivial after reduction modulo $\lambda$.
\end{enumerate}

We now specify our open compact subgroup $U$ as follows:
\begin{enumerate} \item $U_v = G(\cO_{L_v^+})$ if $v \in S_l$.
 \item $U_v = G(\cO_{L_v^+})$ if $v \not\in T$ is split in $L$.
\item $U_v$ is a hyperspecial maximal compact of $G(L_v^+)$ if $v$ is inert in $L$.
\item $U_v = \Iw(\wv)$ for $v \in R$.
\item $U_{v_1} = \Iw(\wv_1).$
\end{enumerate}

Then for all $t \in G(\bbA_{L^+}^\infty)$, the group $t^{-1} G(L^+) t \cap U$ contains no element of order $l$ (by the choice of $v_1$). Let $\ffrm = \ffrm_{\{1\}}$ be a non-Eisenstein maximal ideal of  $\bbT^{T, \text{ord}}_{\{1\}}(U(\frl^\infty), \cO)$. Using that the $\chi_v$ are trivial modulo $\lambda$, we can identify the spaces
\begin{equation*}S^\text{ord}_{0, \{\chi_v\}}(U(\frl^\infty), k) = S^\text{ord}_{0, \{1\}}(U(\frl^\infty), k).\end{equation*}
Thus we get a maximal ideal $\ffrm_{\{\chi_v \}}$ of the algebra $\bbT^{T, \text{ord}}_{\{\chi_v\}}(U(\frl^\infty), \cO)$.

We will define a deformation problem, by giving the local lifting ring $R_v$ corresponding to a local deformation problem $\cD_v$ for each $v \in T$. We are going to consider liftings to the category $\cC_\Lambda$ of complete Noetherian local $\Lambda$-algebras with residue field $k$, instead of the category $\cC_\cO$ used above. See \cite{Ger09},  Definition 4.1.3. This is a minor technical point which will not appear again, so we make no further mention of it. 

For $v \in R$, let $R_v = R_\wv^\chi$. For $v \in S_l$, we will take $R_v = R_{\Lambda_\wv}^{\triangle, ar}$. We take $R_{v_1} = R^a_{\wv_1}$. (See Section \ref{galoisdef} for the definitions of these quotients).

We now have a deformation problem $\cS_{\{\chi_v\}}$ given by the tuple
\begin{equation*} \left(L/L^+, T, \widetilde{T}, \Lambda, \overline{r}_{\ffrm_{\{\chi_v\}}} , \epsilon^{1-n} \delta^{\mu_\ffrm}_{L/L^+}, \{R_\wv^\chi\}_{v \in R} \cup \{R_{\Lambda_\wv}^{\triangle, ar}\}_{v \in S_l} \cup \{ R_{\wv_1}^a\}\right).\end{equation*}
We set $R^\text{loc}_{\{\chi_v\}} =  \left( \widehat{\bigotimes}_{v \in R} R_\wv^\chi\right) \widehat{\otimes}_\cO\left(\widehat{\bigotimes}_{v \in S_l} R_{\Lambda_\wv}^{\triangle, ar} \right) \widehat{\otimes}_\cO R^a_{\wv_1}.$
Using the natural isomorphism $\widehat{\otimes}_{v \in S_l} \Lambda_\wv \cong \Lambda$, we see that $R^\text{loc}_{\{\chi_v\}}$ is naturally a $\Lambda$-algebra.

\begin{theorem}\label{taylorwilesord} With assumptions as above, the lifting
\begin{equation*}r_{\ffrm_{\{\chi_v\}}} : G_{L^+} \rightarrow \cG_n(\bbT^{T, \text{ord}}_{\{\chi_v\}}(U(\frl^\infty), \cO)_{\ffrm_{\{\chi_v\}}})\end{equation*}
defined above is of type $\cS_{\{\chi_v\}}$. Suppose moreover that $\overline{r}_\ffrm(G_{L^+(\zeta_l)})$ is adequate. Let $r : G_{L^+} \rightarrow \cG_n(\cO)$ be a lifting of $\overline{r}_\ffrm$ of type $\cS_{\{1\}}$, which is unramified above $v_1$. Suppose there exists a homomorphism $f' : \bbT^{T, \text{ord}}_{\{1\}}(U(\frl^\infty), \cO)_\ffrm \rightarrow \cO$ such that $(f' \circ r_\ffrm)$ is unramified above $v_1$.

Then there exists a second homomorphism $f : \bbT^{T, \text{ord}}_{\{1\}}(U(\frl^\infty), \cO)_\ffrm \rightarrow \cO$ such that $r$ and $f \circ r_\ffrm$ are conjugate by an element of $\GL_n(\cO)$.

\end{theorem}

\begin{proof}
That $r_{\ffrm_{\{\chi_v\}}}$ is of type $\cS_{\{\chi_v\}}$ is Lemma 4.1.7 of \cite{Ger09}. To prove the above theorem we will apply the Taylor-Wiles-Kisin method in a similar manner to before. Let
\begin{equation*}q_0 = [L^+ : \bbQ]n(n-1)/2 + [L^+ : \bbQ]n(1- (-1)^{\mu_\ffrm - n})/2.\end{equation*}
We can fix for every $N \geq 1$ a tuple ($Q_N, \widetilde{Q}_N, \{\psi_{\wv}\}_{v \in Q_N}$) such that \begin{itemize}
\item $\#Q_N = q$ for all $N$.
\item $\bbN v \equiv 1 \mod l^N$ for all $v \in Q_N$.
\item The ring $R^{\square_T}_{\cS_{\{\chi_v\}, Q_N}}$ can be topologically generated over $R^\text{loc}_{\{\chi_v\}}$ by
\begin{equation*}g = q-q_0 = q - [L^+ : \bbQ]n(n-1)/2 + [L^+ : \bbQ]n(1- (-1)^{\mu_\ffrm -n})/2\end{equation*}
elements.
\item $\psi_{\wv}$ is a subrepresentation of $\overline{r}|_{G_{L_\wv}}$ of dimension $d_N^v$.
\end{itemize}
This can be deduced from Proposition \ref{Chebotarev} just as Proposition 4.2.1 of \emph{op. cit.} is deduced from Proposition 2.5.9 of \cite{Clo08}. For each $v, N \geq 1$, let $\frp_N^\wv,\frp_{N,1}^\wv$ be the parahoric subgroup of $\GL_n(\cO_{L_\wv})$ corresponding to the partition $n = (n - d_N^v) + d_N^v$, as defined previously. We let
\begin{equation*}U_1(Q_N) = \prod_v U_1(Q_N)_v, U_0(Q_N) = \prod_v U_0(Q_N)_v\end{equation*}
be the compact open subgroups of $G(\bbA^\infty_{L^+})$ with $U_i(Q_N) = U_v$ for $v \not\in Q_N$, $U_1(Q_N) =  \iota_{\wv}^{-1} \frp_{N,1}^\wv$, $U_0(Q_N) = \iota_{\wv}^{-1} \frp_{N}^\wv$ for $v \in Q_N$. 

We have natural maps
\begin{equation*}\bbT^{T\cup Q_N,\text{ord}}_{\{\chi_v\}}(U_1(Q_N)(\frl^\infty), \cO) \rightarrow \bbT^{T\cup Q_N,\text{ord}}_{\{\chi_v\}}(U_0(Q_N)(\frl^\infty), \cO)\end{equation*}
\begin{equation*}\rightarrow \bbT^{T \cup Q_N, \text{ord}}_{\{\chi_v\}}(U(\frl^\infty), \cO) \rightarrow \bbT^{T, \text{ord}}_{\{\chi_v\}}(U(\frl^\infty), \cO).\end{equation*}
The first two arrows are surjective, the last is injective. Thus ${\ffrm_{\{\chi_v\}}}$ determines maximal ideals denoted $\ffrm_{\{\chi_v\}, Q_N}, \ffrm_{\{\chi_v\}, Q_N}, {\ffrm_{\{\chi_v\}}}, {\ffrm_{\{\chi_v\}}}$ of these four algebras.
After localizing at ${\ffrm_{\{\chi_v\}}}$ the last map is an isomorphism, since $\overline{r}_\ffrm$ is absolutely irreducible (cf. the proof of \cite{Clo08}, Corollary 3.4.5). We set $\bbT_{\{\chi_v\}} = \bbT^{T, \text{ord}}_{\{\chi_v\}}(U(\frl^\infty), \cO)_{\ffrm_{\{\chi_v\}}}$.

For each $v \in Q_N$, let $\phi_\wv \in G_{L_\wv}$ be a Frobenius lift and let $\varpi_\wv$ be the uniformizer with $\Art_{L_\wv}\varpi_\wv = \phi_\wv |_{L_\wv^\text{ab}}$. In Proposition \ref{mainprop} we have defined a projection operator $\pr_{\varpi_\wv}$ for each $v \in Q_N$. For each $\alpha \in \cO_{L_\wv}^\times$, we have the Hecke operator $V_\alpha$ defined above.

Let $\Delta_{Q_N} = U_0(Q_N)/U_1(Q_N)$. Let $\fra_{Q_N}$ denote the kernel of the augmentation map $\Lambda[\Delta_{Q_N}] \rightarrow \Lambda$. For $i=0, 1$, let $H_{i, Q_N}$ be defined by
\begin{equation*}H_{i, \{\chi_v\}, Q_N}^\vee = \left(\prod_{v \in Q_N} \pr_{\varpi_\wv}\right) S^\text{ord}_{0, \{\chi_v\}}(U_i(Q_N)(\frl^\infty), K/\cO)_{\ffrm_{\{\chi_v\}, Q_N}},\end{equation*}
where $(-)^\vee = \Hom_\cO( -, K/\cO)$ denotes Pontryagin dual. Let $\bbT_{i, \{\chi_v\}, Q_N}$ denote the image of $\bbT^{T\cup Q_N,\text{ord}}_{\{\chi_v\}}(U_i(Q_N)(\frl^\infty), \cO)$ in $\End_\Lambda(H_{i, \{\chi_v\}, Q_N}).$ Similarly we define $H_{\{\chi_v\}}$ by
\begin{equation*}H_{\{\chi_v\}}^\vee =  S^\text{ord}_{0, \{\chi_v\}}(U(\frl^\infty), K/\cO)_{\ffrm_{\{\chi_v\}, Q_N}}.\end{equation*}

 We claim that the following hold:
\begin{enumerate} \item For each $N$, the map
\begin{equation*}\left( \prod_{v \in Q_N} \pr_{\varpi_\wv}  \right)^\vee: H_{0, \{\chi_v\}, Q_N}\rightarrow H_{\{\chi_v\}}\end{equation*}
is an isomorphism.
\item For each $N$, $H_{1, \{\chi_v\}, Q_N}$ is free over $\Lambda[\Delta_{Q_N}]$ with
\begin{equation*}H_{1, \{\chi_v\}, Q_N}/\fra_{Q_N} \overset{\sim}{\rightarrow} H_{0, \{\chi_v\}, Q_N},\end{equation*}
the isomorphism induced by restriction.
\item For each $N$ and for each $v \in Q_N$, there is a character $V_\wv : \cO_{L_\wv}^\times \rightarrow \bbT_{1, \{\chi_v\}, Q_N}^\times$ such that
\begin{enumerate} \item for any $\alpha$, $V_\alpha = V_\wv(\alpha)$ on $H_{1, \{\chi_v\}, Q_N}$.
\item ($r_{\ffrm_{\{\chi_v\}, Q_N}} \otimes \bbT_{1, \{\chi_v\}, Q_N})|_{G_{L_\wv}} \cong s \oplus \psi$ with $s$ unramified, lifting $\overline{s}_\wv$, $\psi$ lifting $\overline{\psi}_\wv$, and $\psi|_{I_{L_\wv}}$ acting as the scalar character $V_\wv \circ \Art^{-1}_{L_\wv}|_{I_{L_{\wv}}}$.
\end{enumerate}
\end{enumerate}
These are proved as in Lemma 4.2.2 of \cite{Ger09}, by passing to the limit from corresponding facts in the finite level case and using Proposition \ref{free} above. 

For each $N$, the lift $r_{\ffrm_{\{\chi_v\}, Q_N}} \otimes \bbT_{1, \{\chi_v\}, Q_N}$ is of type $\cS_{\{\chi_v\}, Q_N}$ and gives rise to a surjection $R^\text{univ}_{\cS_{\{\chi_v\}, Q_N}} \rightarrow \bbT_{1,\{\chi_v\}, Q_N}$. Thinking of $\Delta_{Q_N}$ as the maximal $l$-power order quotient of $\prod_{v \in Q_N} I_{L_\wv}$, we obtain a homomorphism $\Delta_{Q_N} \rightarrow (R^\text{univ}_{\cS_{\{\chi_v\}, Q_N}})^\times$ as follows: for each $v \in Q_N$, decompose $r^\text{univ}_{\cS_{\{\chi_v\}, Q_N}}|_{G_{L_\wv}} = s \oplus \psi$ and take a diagonal entry of $\psi$ in some basis. This does not depend on the choice of basis.

We thus have homomorphisms $\Lambda[\Delta_{Q_N}] \rightarrow R^\text{univ}_{\cS_{\{\chi_v\}, Q_N}} \rightarrow R^{\square_T}_{\cS_{\{\chi_v\}, Q_N}}$ and natural isomorphisms $R^\text{univ}_{\cS_{\{\chi_v\}, Q_N}}/\fra_{Q_N} \cong R^\text{univ}_{\cS_\{\chi_v\}}$ and $R^{\square_T}_{\cS_{\{\chi_v\}, Q_N}}/\fra_{Q_N} \cong R^{\square_T}_{\cS_{\{\chi_v\}}}$. Moreover, this makes the maps $R^\text{univ}_{\cS_{\{\chi_v\}, Q_N}} \rightarrow \bbT_{1,\{\chi_v\},Q_N}$ into homomorphisms of $\Lambda[\Delta_{Q_N}]$-algebras.

At this point we apply a patching argument, for the details of which we refer to the proof of Theorem 4.3.1 of \cite{Ger09}. For each $v \in R$, we choose a character $\chi_v = \chi_{v, 1} \times \dots \times \chi_{v, n}$ such that the $\chi_{v, j}$ are all distinct (this can be done since we assumed that $l^N \| \bbN v - 1 \Rightarrow l^N > n$ for each $v \in R$). We define the following rings:
\begin{equation*}\begin{array}{rcl} \cT & = & \cO[[x_1, \dots, x_{n^2 \#T}]]\\ 
\Delta_\infty & = & \bbZ_l^q \\
S_\infty & = & \Lambda \widehat{\otimes}_\cO \cT[[\Delta_\infty]]\\
R^{\square_T}_{\{\chi_v\},\infty} & = & R^\text{loc}_{\{\chi_v\}}[[Y_1, \dots, Y_{q'}]]\\
R^{\square_T}_{\{1\},\infty} & = & R^\text{loc}_{\{1\}}[[Y_1, \dots, Y_{q'}]]\\
\fra & = & \ker(S_\infty \rightarrow \Lambda).
\end{array}\end{equation*}
After patching, one is given the following:
\begin{enumerate}\item  Surjective homomorphisms
\begin{equation*}R^{\square_T}_{\{\chi_v\},\infty} \rightarrow R^\text{univ}_{\cS_{\{\chi_v\}}} \text{ and }R^{\square_T}_{\{1\},\infty} \rightarrow R^\text{univ}_{\cS_{\{1\}}}.\end{equation*}
\item Modules $H^\square_{1, \{\chi_v\}, \infty}$ for $R^{\square_T}_{\{\chi_v\},\infty}$ and $H^\square_{1, \{1\}, \infty}$ for $R^{\square_T}_{\{1\},\infty}.$ 
\item Commuting actions of $S_\infty$ on these modules, such that they become free of finite rank over $S_\infty$.
\end{enumerate}
These satisfy the following:
\begin{enumerate} \item The $S_\infty$ actions factor through the respective $R^{\square_T}_{\{\chi_v\},\infty}$ and $R^{\square_T}_{\{1\},\infty}$ actions.
\item There are isomorphisms 
\begin{equation*}H^\square_{1, \{\chi_v\}, \infty}/\fra \cong H_{\{\chi_v\}}\end{equation*}
\begin{equation*}H^\square_{1, \{ 1\}, \infty}/\fra \cong H_{\{1\}}\end{equation*}
compatible with the homomorphisms $R^{\square_T}_{\{\chi_v\},\infty} \rightarrow R^\text{univ}_{\cS_{\{\chi_v\}}} \rightarrow \bbT_{\{\chi_v\}}$ and $R^{\square_T}_{\{1\},\infty} \rightarrow R^\text{univ}_{\cS_{\{1\}}} \rightarrow \bbT_{\{1\}}$.
\item After modding out by $\lambda$, the objects decorated with $\{\chi_v\}$ can be identified with the objects decorated with $\{1\}$, and the various maps between them are then the same.
\end{enumerate}
As in the proof of Theorem \ref{taylorwiles}, we deduce that $\mu_\ffrm \equiv n \mod 2$ and that for every minimal prime $Q$ of $\Lambda$, the support of $H^\square_{1, \{\chi_v\}, \infty}/Q$ in $\Spec R^{\square_T}_{\{\chi_v\},\infty}/Q$ is a union of irreducible components. The same statement holds for the support of $H^\square_{1, \{1\}, \infty}/Q$ in $\Spec R^{\square_T}_{\{1\},\infty}/Q.$ 

By Lemma 3.3 of \cite{Bar09}, giving an irreducible component $\cC$ of $R^{\square_T}_{\{\chi_v\},\infty}$ or  $R^{\square_T}_{\{1\},\infty}$ is the same as giving an irreducible component $\cC_v$ of the local lifting ring $R_v$ for each $v \in T$. We will write suggestively $\cC = \otimes_v \cC_v$.  

Take a minimal prime $Q$ of $\Lambda$, and consider $R^{\square_T}_{\{\chi_v\},\infty}/Q$. It satisfies:
\begin{enumerate}  \item Every generic point of $\Spec R^{\square_T}_{\{\chi_v\},\infty}/Q$ has characteristic 0.
\item For each $v \neq v_1$, $R_v$ is irreducible. In particular, there is a bijection between the irreducible components of $R^{\square_T}_{\{\chi_v\},\infty}/Q$ and the irreducible components of $R_{v_1}$.
\end{enumerate}

On the other hand, $R^{\square_T}_{\{1\},\infty}/Q$ satisfies the following properties:
\begin{enumerate}  \item Every generic point of $\Spec R^{\square_T}_{\{1\},\infty}/Q$ has characteristic 0.
\item Every prime of $R^{\square_T}_{\{1\},\infty}/Q$ minimal over $\lambda$ contains a unique minimal prime.
\end{enumerate}
Justification for the fact that the individual factors of $R^\text{loc}_{\{\chi_v\}}$ and $R^\text{loc}_{\{1\}}$ have the relevant properties has been given in Section \ref{galoisdef}. One now just needs to know that they are preserved under completed tensor products, and this is the content of Lemma 3.3 of \cite{Bar09}.

We now argue as follows. Using the existence of $f$, we see that $\Supp_{ R^{\square_T}_{\{1\},\infty}}H^\square_{1, \{1\}, \infty}$ contains an irreducible component $\cC = \otimes_v \cC_v$ with $\cC_{v_1} = \cC_{v_1}^{ur}$, where $\cC_{v_1}^{ur}$ is the irreducible component of $R^a_{\wv_1}$ which classifies unramified liftings. (The local component at $v_1$ is generic). 

Using the identification modulo $\lambda$, we see that \begin{equation*}\Supp_{R^{\square_T}_{\{\chi_v\},\infty}/\lambda}H^\square_{1, \{\chi_v\}, \infty}/\lambda\end{equation*} contains an irreducible component $\otimes_v \overline{\cC}_v$ of $\Spec R^{\square_T}_{\{\chi_v\},\infty}/\lambda$ with $\overline{\cC}_{v_1} = \overline{\cC}_{v_1}^{ur}$, in the obvious notation. Using the properties (1) and (2) of $\Spec R^{\square_T}_{\{\chi_v\},\infty}$ above, we see that $\Supp_{R^{\square_T}_{\{\chi_v\},\infty}}H^\square_{1, \{\chi_v\}, \infty}$ contains an irreducible component $\cC'_v = \otimes_v \cC'_v$ with $\cC'_{v_1} = \cC_{v_1}^{ur}$. (Here we use the following property of $R^a_{\wv_1}$: every prime of $R^a_{\wv_1}$ minimal over $\lambda$ contains a unique minimal prime.)

Applying the identification modulo $\lambda$ in the opposite direction, we find that
\begin{equation*}\Supp_{R^{\square_T}_{\{1\},\infty}/\lambda}H^\square_{1, \{1\}, \infty}/\lambda\end{equation*}
contains every irreducible component $\overline{\cC}' = \otimes_v \overline{\cC}'_v$ with $\overline{\cC}'_{v_1} = \overline{\cC}_{v_1}^{ur}$. Now applying properties (1) and (2) of $\Spec R^{\square_T}_{\{1\},\infty}$ above, we deduce that for every minimal prime $Q$ of $\Lambda$, $\Supp_{ R^{\square_T}_{\{1\},\infty}/Q}H^\square_{1, \{1\}, \infty}/Q$ contains every irreducible component $\cC'' = \otimes_v{\cC''_v}$ with $\cC''_{v_1} = \cC_{v_1}^{ur}$.

Now consider the deformation problem
\begin{equation*}\cS'_{\{ 1 \}}=\left(L/L^+, T, \widetilde{T}, \Lambda, \overline{r}_{\ffrm} , \epsilon^{1-n} \delta^{\mu_\ffrm}_{L/L^+}, \{R_\wv^1\}_{v \in R} \cup \{R_{\Lambda_\wv}^{\triangle, ar}\}_{v \in S_l} \cup \{ R_{\wv_1}^{{ur}} )\}\right).\end{equation*}
As in the proof of Theorem \ref{taylorwiles}, we see that $H_{\{1\}}  \otimes_{R^\text{univ}_{\cS_{\{1\}}}} R^\text{univ}_{\cS'_{\{1\}}}$ is a nearly faithful $R^\text{univ}_{\cS'_{\{1\}}}$-module, hence the homomorphism $R^\text{univ}_{\cS'_{\{1\}}} \rightarrow \cO$ induced by $r'$ actually factors through $\bbT_{\{1\}}$. (Recall that $\bbT_{\{1\}}$ is reduced). This concludes the proof of the theorem.
\end{proof} 

\begin{corollary}\label{ordinaryfinite} With the hypotheses of the theorem, $R^\text{univ}_{\cS'_{\{1\}}}$ is a finite $\Lambda$-module.
\end{corollary}
\begin{proof}
This is proved in exactly the same way as Corollary \ref{minimalfinite}, using that $\bbT_{\{1\}}$ is a finite $\Lambda$-algebra.
\end{proof}

 \section{Ordinary automorphy lifting}
 
\begin{theorem}\label{ordinary} Let $F$ be an imaginary CM field with totally real subfield $F^+$ and let $c$ be the non-trivial element of $\Gal(F/F^+)$. Let $n \in \bbZ_{\geq 1}$ and let $l$ be an odd prime. Let $K \subset \overline{\bbQ}_l$ denote a finite extension of $\bbQ_l$ with ring of integers $\cO$, residue field $k$ and maximal ideal $\lambda$. Let
 \begin{equation*} \rho : G_F \rightarrow \GL_n(\cO)\end{equation*}
 be a continuous representation and let $\overline{\rho} = \rho \mod \lambda$. Let $\mu : G_{F^+} \rightarrow \cO^\times$ be a continuous character. Suppose that $\rho$ satisfies the following properties:
 \begin{enumerate} \item $\rho^c \cong \rho^\vee \epsilon^{1-n} \mu|_{G_F}.$
 \item $\mu(c_v)$ is independent of $v \mid \infty$.
 \item $\rho$ is ramified at only finitely many places.
 \item For each $v \mid l$, $\rho|_{G_{F_v}}$ is ordinary, in the sense of Definition \ref{ord}.
 \item $\overline{\rho}$ is absolutely irreducible and $\overline{\rho}(G_{F(\zeta_l)}) \subset \GL_n(k)$ is adequate.
 \item $\zeta_l \not\in F$.
 \item There is a continuous representation $\rho' : G_F \rightarrow \GL_n(\cO)$, a continuous character $\mu' : G_{F^+} \rightarrow \cO^\times$, a RAECSDC automorphic representation $(\pi, \chi)$ of $\GL_n(\bbA_F)$ and $\iota : \overline{\bbQ}_l \overset{\sim}{\rightarrow} \bbC$ such that 
 \begin{enumerate} 
 \item $\rho' \otimes_\cO \barQl \cong r_{l, \iota}(\pi) : G_F \rightarrow \GL_n(\barQl).$
 \item $\mu' \otimes_\cO \barQl \cong r_{l, \iota}(\chi).$
 \item $(\overline{\rho},\overline{\mu}) = (\overline{\rho}',\overline{\mu}')$.
 \item $\pi$ is $\iota$-ordinary at all places dividing $l$.
 \end{enumerate}
 \end{enumerate}
 Then $\rho$ is ordinarily automorphic. If moreover $\rho$ is crystalline (resp. potentially crystalline) at each place of $F$ above $l$ then $\rho$ is ordinarily automorphic of level prime to $l$ (resp. of level potentially prime to $l$).
 \end{theorem} 
 \begin{proof} As in the proof of Theorem \ref{minimal}, we reduce to the case that $\chi = \mu = \delta^n_{F/F^+}$. The reduction of the theorem in this case to Theorem \ref{taylorwilesord}, proved above, is analogous to the reduction of Theorem 5.3.2 of \cite{Ger09} to Theorem 4.3.1 of that paper. The only changes are as follows. First, instead of introducing the places $S_a$, one chooses a place $v$ of $F^+$ split in $F$ such that $\bbN v \not\equiv  1 \mod l$ and such that both $\rho$ and $\pi$ are unramified above $v$. After making a soluble base change one takes the place $v_1$ of Theorem \ref{taylorwilesord} to be a place of $L$ above $v$. The choice of $U_{v_1}$ there ensures that for all $t \in G(\bbA_{L^+}^\infty)$, the group $t^{-1} G(L^+) t \cap U$ contains no element of order $l$. Second, one has to ensure that if $l^N \| (\bbN v - 1)$ then $l^N > n$ at the primes at which $\rho$ or $\pi$ ramify. (The final sentence of the theorem requires a fixed weight version of Theorem \ref{taylorwilesord}, as in Theorem 4.3.1 of \cite{Ger09}. The modifications that one must make to the proof in this case are analogous and slightly simpler to those made for the variable weight version, so we omit them here).
\end{proof}

\section{Finiteness theorems}

Let $F$ be an imaginary CM field with totally real subfield $F^+$, and let $l$ be an odd prime. Let $S$ be a finite set of places of $F^+$, each of which splits in $F$, and containing all of the places above $l$, and let $\widetilde{S}$ be a set of places of $F$ containing exactly one place above each place of $S$. Let $K$ be a finite extension of $\bbQ_l$, contained in $\barQl$, with ring of integers $\cO$ and residue field $k$. Fix an isomorphism $\iota : \barQl \overset{\sim}{\rightarrow} \bbC.$

\subsection*{A minimal finiteness theorem}

Let $(\pi, \chi)$ be a RAECSDC automorphic representation of $\GL_n(\bbA_F)$ of weight $\iota_\ast \lambda$, unramified outside $S$. Let $\rho : G_F \rightarrow \GL_n(\cO)$ be a continuous representation with $r_{l, \iota}(\pi) \cong \rho \otimes_\cO \barQl$. Let $\mu$ be a character $G_{F^+} \rightarrow \cO^\times$ with $r_{l,\iota}(\chi) = \mu \otimes_\cO \barQl$. Suppose that $\overline{\rho}$ is absolutely irreducible, and let $\overline{r}$ be an extension to a representation
\begin{equation*}\overline{r} : G_{F^+} \rightarrow \cG_n(k)\end{equation*}
with $\nu \circ \overline{r} = \overline{\mu} \epsilon^{1-n} \delta_{F/F^+}^{\kappa}$ for some $\kappa \in \bbZ/2\bbZ$. For each $v \in S$, fix an irreducible component $\cC_v$ of $R^{\overline{\square}}_\wv$ if $v \nmid l$ or $R^{\lambda, \text{cr}}_\wv$ if $v \mid l$ such that $\rho|_{G_{F_\wv}}$ lies on no other irreducible component. 

We consider the global deformation problem
\begin{equation*}\cS = \left(F/F^+, S, \widetilde{S}, \cO, \overline{r}, \mu \epsilon^{1-n} \delta^\kappa_{F/F^+}, \{ R^{\cC_v}_\wv\}_{v \in S, v \nmid l} \cup  \{ R^{\lambda, \cC_v}_\wv\}_{v \in S, v \mid l}\right).\end{equation*}
\begin{theorem} With hypotheses as above, suppose further that the group $\overline{\rho}(G_{F(\zeta_l)})$ is adequate and that $\zeta_l \not\in F$. Then $\kappa =0$ and $R^\text{univ}_\cS$ is a finite $\cO$-module.
\end{theorem}
\begin{proof}
After twisting as in the proof of Theorem \ref{minimal}, we can assume that $\chi = \delta_{F/F^+}^n$. Let $L/F$ be an imaginary CM extension such that:
\begin{enumerate} \item $L/F$ is Galois and soluble.
\item $L$ is linearly disjoint with $\overline{F}^{\ker \overline{\rho}}(\zeta_l)$ over $F$.
\item 4 divides $[L^+:F^+]$.
\item $L/L^+$ is unramified at all finite places.
\item Each place of $S$ splits completely in $L^+$.
\end{enumerate}
In particular, the hypotheses of section 6 are satisfied and we are given a unitary group $G_{/\cO_{L^+}}$. There exists an automorphic representation $\Pi$ of $G(\bbA_{L^+})$ such that $\pi_L$ is a strong base change of $\Pi$, in the sense of \cite{Lab09}.

Let $S_l^L$, $S_r^L$ denote the sets of places of $L^+$ above $S_l$ and $S_r$, respectively, and let $\widetilde{S}_l^L$ and $\widetilde{S}_r^L$ be defined analogously. Choose a place $v_1$ of $L^+$ split in $L$ at which $\pi$ is unramified, and such that the residue characteristic of $v_1$ does not divide the order of any element of $G(L^+)$. Let $\wv_1$ we a place of $L$ above $v_1$. Let $T = S_l^L \cup S_r^L \cup \{v_1\}$, and $\widetilde{T} = \widetilde{S}_l^L \cup \widetilde{S}_r^L \cup \{ \wv_1 \}.$ We define one further deformation problem $\cS^L$ by the tuple
\begin{equation*} \left( L/L^+, T, \widetilde{T}, \cO, \overline{r}|_{G_{L^+}}, \epsilon^{1-n}\delta^{\kappa+n}_{L/L^+}, \{R_\wv^{\lambda, \cC_v}\}_{v \in S_l^L} \cup \{ R^{\cC_v}_\wv \}_{v \in S_r^L} \cup \{ R_{\wv_1}^{ur} \}\right).\end{equation*}
(Since every prime of $S$ splits in $L^+$, the components $\cC_v$ for $v \in S$ we chose induce components of the lifting rings for $v \in S_L$ in a natural manner).
There is a natural homomorphism $R^\text{univ}_{\cS^L} \rightarrow R^\text{univ}_\cS$, given by restriction of the universal deformation. The argument of Lemma 3.2.5 of \cite{Gee09} shows that this homomorphism is in fact \emph{finite}, so it suffices to show that $R^\text{univ}_{\cS^L}$ is a finite $\cO$-algebra.

Let $U = \prod_v U_v$ be the open compact subgroup of $G(\bbA_{L^+}^\infty)$ defined as follows:
\begin{enumerate} \item $U_v = G(\cO_{L_v^+})$ if $v \in S^L_l$.
 \item $U_v = G(\cO_{L_v^+})$ if $v \not\in T$ is split in $L$.
\item $U_v$ is a hyperspecial maximal compact of $G(L_v^+)$ if $v$ is inert in $L$.
\item $\Pi_v^{U_v} \neq 0$ for $v \in S^L_r$.
\item $U_{v_1} = \iota_{\wv_1}^{-1} \Iw_1(\wv_1)$.
\end{enumerate}

Let $\widetilde{I}_l$ be the set of embeddings $L \hookrightarrow K$ giving rise to a place of $\widetilde{S}^L_l$. Define an element $\lambda_L \in (\bbZ^n_+)^{\widetilde{I}_l}$ by $\lambda_{L, \tau} = \lambda_{\tau|_F}$. Then $(\iota^{-1} \Pi^\infty)^U \neq 0$, and the action of $\bbT^T_{\lambda_L}(U, \cO)$ on $ (\iota^{-1} \Pi^\infty)^U$ gives rise to a homomorphism $\bbT^T_{\lambda_L}(U, \cO) \rightarrow \barQl$. After possibly extending $K$, we can suppose that this homomorphism takes values in $\cO$. Let $\ffrm$ be the maximal ideal contained in the kernel of this homomorphism. We can suppose that the representation $\overline{r}_\ffrm : G_{L^+} \rightarrow \cG_n(k)$ is equal to $\overline{r}|_{G_L^+}$. (Note that their restrictions to $G_L$ are already isomorphic, so this is a matter of choosing the correct extension of $\overline{r}_\ffrm$ to $\cG_n$. See Lemma 2.1.4 of \cite{Clo08}). 

The hypotheses of Theorem \ref{taylorwiles} and its corollary now apply, and the result follows.
\end{proof}

\subsection*{An ordinary finiteness theorem}

Let $(\pi, \chi)$ be a RAECSDC automorphic representation of $\GL_n(\bbA_F)$, unramified outside $S$ and $\iota$-ordinary. Let $\rho : G_F \rightarrow \GL_n(\cO)$ be a continuous representation with $r_{l, \iota}(\pi) \cong \rho \otimes_\cO \barQl$. Let $\mu : G_{F^+} \rightarrow \cO^\times$ be a de Rham character with $\overline{\mu} = \overline{r}_{l, \iota}(\chi)$. Suppose that $\overline{\rho}$ is absolutely irreducible, and let $\overline{r}$ be an extension to a representation
\begin{equation*}\overline{r} : G_{F^+} \rightarrow \cG_n(k)\end{equation*}
with $\nu \circ \overline{r} = \overline{\mu}  \epsilon^{1-n} \delta_{F/F^+}^{\kappa}$ for some $\kappa \in \bbZ/2\bbZ$. Note that $HT_\tau(\mu) = \{ w \}$ is independent of $\tau : F^+ \hookrightarrow K$. Choose $\lambda \in (\bbZ^n_+)^{\Hom(F, \barQl)}_w$. 

We have the global deformation problem
\begin{equation*}\cS = \left(F/F^+, S, \widetilde{S}, \cO, \overline{r}, \mu \epsilon^{1-n} \delta_{F/F^+}^{\kappa}, \{ R^\square_\wv \}_{v \in S, v \nmid l} \cup \{R^{\lambda, \text{ss-ord}}_\wv \}_{v \in S, v \mid l} \right).\end{equation*}
(We use $R^\square_\wv$ here to denote the unrestricted lifting ring at the places in $S$ not dividing $l$).

\begin{theorem} With hypotheses as above, suppose further that the group $\overline{\rho}(G_{F(\zeta_l)})$ is adequate and that $\zeta_l \not\in F$. Then $\kappa = 0$ and $R^\text{univ}_\cS$ is a finite $\cO$-module.
\end{theorem}
\begin{proof} As in the proof of Theorem \ref{minimal}, we can reduce to the case where $\chi = \mu = \delta_{F/F^+}^n$, except that after twisting $\pi$ may be ramified outside $S$. Choose an imaginary CM extension $L/F$ such that
\begin{enumerate} \item $L/F$ is Galois and soluble.
\item $L$ is linearly disjoint with $\overline{F}^{\ker \overline{\rho}}(\zeta_l)$ over $F$.
\item 4 divides $[L:F]$.
\item $L/L^+$ is unramified at all finite places.
\item Each place of $S$ splits completely in $L^+$.
\item $\pi$ is unramified outside the places of $L$ above $S$, and has an Iwahori fixed vector at all finite places of $L$.
\item For each place $v$ of $L$ lying above $S$, $\overline{r}|_{G_{L_v}}$ is trivial.
\item If $v \nmid l$ is a place of $L$ above $S$ then $ \bbN v \equiv 1 \mod l$, and if $l^N \| \bbN v -1$ then $l^N > n$.
\end{enumerate}
Let $S_L$ denote the set of places of $L^+$ lying above a place of $S$, and define $\widetilde{S}_L$ analogously. For each place $v \nmid l$ in $S_L$ there exists a finite extension $M_\wv$ of $L_\wv$ such that for any lift $\tau$ of $\overline{r}|_{G_{L_\wv}}$, $\tau|_{G_{I_{M_\wv}}}$ acts unipotently. We can find a soluble CM extension $M/L$ such that $M$ satisfies all the properties above, and the following:
\begin{itemize} \item Let $\widetilde{w}$ be a place of $M$ above a place $\wv$ of $\widetilde{S}_L$. Then $M_{\widetilde{w}} \supset M_\wv$.
\end{itemize}
Choose an auxiliary prime $v_1$ of $M^+$, split in $M$, such that $\Iw(v_1)$ contains no elements of order $l$. Let $S_M$ be the set of places of $M^+$ above $S$, and let $T = S \cup \{v_1\}$. Choose a place $\wv_1$ of $M$ above $v_1$, and define $\widetilde{S}_M$ and $\widetilde{T}$ similarly. We now define two global deformation problems \[\cS^M_\lambda = \left(M/M^+, T, \widetilde{T}, \cO, \overline{r}|_{G_{M^+}},\epsilon^{1-n}\delta^{\kappa+n}_{M/M^+},\right.\] \[ \left. \{R^1_\wv\}_{v \in S_M, v\nmid l} \cup \{R^{\lambda_M,\text{ss-ord}}_\wv\}_{v \in S_M, v \mid l} \cup R^{ur}_{\wv_1}\right).\]
and
\[\cS^M = \left(M/M^+, T, \widetilde{T}, \Lambda, \overline{r}|_{G_{M^+}}, \epsilon^{1-n}\delta^{\kappa+n}_{M/M^+}, \right.\] \[ \left.\{R^1_\wv\}_{v \in S_M, v\nmid l} \cup \{R^{\triangle, ar}_{\Lambda_\wv}\}_{v \in S_M, v \mid l} \cup R^{ur}_{\wv_1}\right).\]
By construction, there is a map $R^\text{univ}_{\cS_\lambda^M} \rightarrow R^\text{univ}_\cS$, given by restriction of the universal deformation, and the argument of Lemma 3.2.5 of \cite{Gee09} shows that it is finite. Thus it suffices to show that $R^\text{univ}_{\cS_\lambda^M}$ is a finite $\cO$-algebra. 

Let $\chi^{\lambda_M} : \Lambda \rightarrow \cO$ be the $\cO$-algebra homomorphism induced by the tuple of characters $(\chi_1, \dots, \chi_n)$ given by
\begin{equation*}\chi_i : u \mapsto \prod_{\tau : M \hookrightarrow K} \tau(u)^{1 - i + \lambda_{M, \tau, n - i + 1}}.\end{equation*}

Let $\wp_\lambda$ denote the kernel of $\chi^{\lambda_M}$. Then $R^\text{univ}_{\cS_\lambda^M} $ is a quotient of $R^\text{univ}_{\cS^M} \otimes_\Lambda \Lambda/\wp_\lambda$. But Corollary \ref{ordinaryfinite} shows that $R^\text{univ}_{\cS^M}$ is a finite $\Lambda$-module, and the result follows.
 \end{proof}

\bibliographystyle{alpha}
\bibliography{ModLiftBib}

\end{document}